\documentclass[a4paper, twoside, 11pt]{amsart}
\usepackage[margin=1 in]{geometry}                
\usepackage{mathtools}
\mathtoolsset{showonlyrefs,showmanualtags}
\usepackage[colorlinks,linkcolor=blue,citecolor=blue,urlcolor=blue]{hyperref}
\usepackage{graphicx}
\usepackage{amssymb, amsmath}
\usepackage{soul}
\usepackage[usenames,dvipsnames]{xcolor}
\usepackage{bbm}
\usepackage{tikz-cd}
\usetikzlibrary{shapes.misc}
\usetikzlibrary{decorations.markings}
\usepackage{yfonts}
\usepackage{etoolbox}
\usepackage{stackengine}[2013-10-15]
\usepackage{mdframed}
\usepackage{bm}
\usepackage[square,numbers]{natbib}
\usepackage{mathdots}
\usepackage{arydshln}
\usetikzlibrary{3d,calc}
\usepackage[normalem]{ulem}

\numberwithin{equation}{section}

\newcommand{\N}{\mathbb N}
\newcommand{\R}{\mathbb R}
\newcommand{\C}{\mathbb C}

\makeatletter
\newcommand*\bigcdot{\mathpalette\bigcdot@{.7}}
\newcommand*\bigcdot@[2]{\mathbin{\vcenter{\hbox{\scalebox{#2}{$\m@th#1\bullet$}}}}}
\makeatother

\usepackage{todonotes}  

\usepackage{tikz}

\newtheorem{theorem}{Theorem}[section]
\newtheorem{lemma}[theorem]{Lemma}

\newtheorem{proposition}[theorem]{Proposition}
\newtheorem{corollary}[theorem]{Corollary}
\newtheorem{remark}[theorem]{Remark}

\begin{document}
\title [Pairs consisting of a Hermitian form and  a self-adjoint antilinear operator]{A canonical form for  pairs consisting of a Hermitian form and  a  self-adjoint antilinear operator}
\date{\today}
\thanks{I.\ Zelenko is supported by Simons Foundation Collaboration Grant for Mathematicians 524213.}

\author{David Sykes}
\address{David Sykes,
	Department of Mathematics
	Texas A\&M University
	College Station
	Texas, 77843
	USA}\email{ dgsykes@math.tamu.edu}
\urladdr{\url{http://www.math.tamu.edu/~dgsykes}}

\author{Igor Zelenko}
\address{Igor Zelenko, Department of Mathematics
	Texas A\&M University
	College Station
	Texas, 77843
	USA}\email{ zelenko@math.tamu.edu}
\urladdr{\url{http://www.math.tamu.edu/~zelenko}}

\subjclass[2010]{15A21, 15A24, 15B05, 15A22,  32V40}
\keywords{antilinear operators, indefinite Hermitian forms, canonical forms, uniformly Levi degenerate CR structures, pencils}	
	\begin{abstract} Motivated by a problem in local differential geometry of Cauchy--Riemann (CR) structures of hypersurface type, we find  a canonical form for pairs consisting of a nondegenerate Hermitian form and a self-adjoint antilinear operator, or, equivalently,  consisting of a nondegenerate Hermitian form and  a symmetric bilinear form. This generalizes the only previously known results on simultaneous normalization of such pairs, namely, the results of \cite{benedetti1984simultaneous}  on simultaneous diagonalization of these pairs in the case where the  Hermitian form is positive definite and of \cite{hong1986reduction}, where a criterion for simultaneous diagonalization is given. 
	\end{abstract}
	
\maketitle

\section{Introduction}\label{fake section place holder}

In the present paper we find canonical forms for pairs consisting of a nondegenerate Hermitian form $\ell$ on a complex $n$-dimensional vector space $W$ and an antilinear operator $A:W\rightarrow W$ that is self-adjoint with respect to the form $\ell$. By a canonical form, as usual, we mean a specified choice of matrices representing elements of any such pair, chosen from among matrix representations in all possible bases of $W$. Our main result is formulated in Theorem \ref{simultaneous canonical form theorem}.  Recall that a map  $A: W\rightarrow W$  is called an antilinear operator if
\[
A(\lambda v+w)=\overline{\lambda}A(v)+A(w) \quad\quad\forall v,w\in W,\lambda\in\C,
\]
and an antilinear operator $A$ is called  self-adjoint with respect to the form $\ell$ or, shortly,  $\ell$-self-adjoint if \begin{equation}
\label{seldef}
\ell(Av,w)=\ell(Aw,v)\quad\quad\forall v,w\in W.
\end{equation}

Our original motivation for this work comes from the local differential geometry of certain Cauchy-Riemann (CR) structures, more precisely of real hypersurfaces of a complex space having  uniformly degenerate Levi form with one dimensional kernel.  As was shown recently in \cite{porter2017absolute}, the basic invariant of such structures at a point is given exactly by a pair of the algebraic objects under consideration. For more details see section \ref{applications} below.
 
Our main result, Theorem \ref{simultaneous canonical form theorem}, also gives  canonical forms for pairs consisting of a nondegenerate Hermitian form and a symmetric bilinear form because the set of these pairs is in bijective correspondence with the one we originally considered. Indeed, to the pair $(\ell, A)$ we can assign the pair $(\ell, \ell')$, where 
\begin{equation}
\label{symmdef}
\ell^\prime(v,w):=\ell(w,Av)
\end{equation}
is a symmetric bilinear form by \eqref{seldef}. From the nondegenericity of $\ell$ it follows that the assignment of $(\ell, A)$  to $(\ell, \ell')$ defines the  bijection between the two sets of pairs under consideration.

Surprisingly, when we encountered the necessity of finding the canonical forms for pairs $(\ell, A)$ in the course of our study in CR geometry, we were not able to find the desired results in the literature. The only results in this direction that we found are those addressing the problem of simultaneous diagonalization \cite{benedetti1984simultaneous, hong1986reduction} and those giving canonical forms for a single antilinear operator \cite{haant, hongthesis, hong1988canonical} and, more generally, for a single semi-linear operator \cite{asano, jacobson} or for a square matrix under $\varphi$-equivalence \cite{hong1991}. In \cite[Theorem 7]{benedetti1984simultaneous}, it shown that $\ell$ and  $A$ can be simultaneously diagonalized if $\ell$ is positive definite, and, in \cite[Theorem 2.1]{hong1986reduction}, the pairs $(\ell, A)$ admitting a simultaneous diagonalization are classified. Perhaps, the main difficulty here is that the matrix representations for a Hermitian form and an antilinear operator transform differently under a change of the basis (see formulas \eqref{transformation} below). It also cannot be reduced to the study of canonical forms of pairs of other objects, wherein the matrix representations of each component of the new pairs transforms in the same way under a basis change. An example of the latter reduction is the set of pairs consisting of a nondegenerate Hermitian form $\ell$ and an  $\ell$-self-adjoint linear operator that was treated in \cite [Theorem 5.1.1]{gohberg2006indefinite} where a canonical form for such pairs is given, which we will refer to as the Gohberg--Lancaster--Rodman form. Although the matrix representations of each component in such pairs transform differently under a basis change, using a process similar to the one in the previous paragraph, we can obtain a bijective correspondence between the set of such pairs and the set of pairs of Hermitian forms (i.e., a pair of the same type of objects), one of which is nondegenerate. In our case, however, such a reduction is not possible and the problem of finding canonical forms cannot be totally reduced to the study of certain classes of matrix pencils, as was classically done using Weierstrass--Kronecker normal forms for matrix pencils (see, for example, \cite{gantmakher1953theory} and \cite{thompson}).

To prove Theorem \ref{simultaneous canonical form theorem}, we develop in section \ref{antilinear operators special cases} a geometric  version of the construction of the canonical form for a single  antilinear operator of \cite{hongthesis} (which was formulated in \cite[Theorem 3.1]{hong1988canonical}, proved in \cite{hongthesis},  and stated for completeness in Remark \ref{Hong--Horn} below) and combine it with a simultaneous normalization of the Hermitian form, which is comparable in certain respects to the method of \cite [subsection 5.3]{gohberg2006indefinite} for obtaining the Gohberg--Lancaster--Rodman form, mentioned in the previous paragraph. By a  geometric version we are referring to the study  of flags of subspaces analogous to the generalized eigenspaces in the standard theory of linear operators as opposed to the algebraic version in \cite{asano, haant, jacobson}  based on the theory of invariant factors  and manipulations with matrices as in \cite{hongthesis, hong1991}.  Our Theorem \ref{simultaneous canonical form theorem} is related to the Hong--Horn canonical form  of \cite[Theorem 3.1]{hong1988canonical} 
for a single  antilinear operator in the same way that the Gohberg--Lancaster--Rodman form in \cite [Theorem 5.1.1]{gohberg2006indefinite} is related to the classical Jordan normal form for linear operators.

In section \ref{alter_section}, for completeness we sketch an alternative approach to the considered problem that leads to an equivalent canonical form, Theorem \ref{alt simultaneous canonical form theorem}. This approach was in fact our original one before we found the more natural and apparently more simple  approach leading to Theorem \ref{simultaneous canonical form theorem}. The idea in this alternative approach is as follows: 
Since $A^2$ is  an  $\ell$-self-adjoint linear operator whenever $A$ is  an  $\ell$-self-adjoint antilinear operator, one can  first bring the pair $(\ell, A^2)$ to the Gohberg--Lancaster--Rodman form and then find a canonical form for $A$ with minimal changes in the form of $\ell$. This requires solving a certain nonlinear matrix equation, which turned out to be feasible.

\section{The Canonical Form}\label{The canonical forms section}
\label{main_result_sec}
As in the introduction, $\ell$ denotes a nondegenerate  Hermitian form and  $A$ denotes an antilinear operator on an $n$-dimensional complex space  $W$.
Unless otherwise stated, $A$ is assumed to be $\ell$-self-adjoint (see \eqref{seldef} for the definition).

Choosing a basis $\{e_1,\ldots,e_n\}$ of $W$, one can represent the form $\ell$ and the antilinear operator $A$ 
by $n\times n$ matrices $H=(H_{i,j})$ and $C=(C_{i,j})$ via a standard construction, requiring, for all $i,j\in\{1,\ldots, n\}$, that
\[
H_{i,j}=\ell(e_j,e_i)\quad\mbox{ and }\quad A(e_i)=\sum_{k=1}^nC_{k,i}e_k .
\]
The conditions that  $\ell$ is a nondegenerate Hermitian form and $A$ is an $\ell$-self-adjoint antilinear operator are equivalent to 
\begin{equation}
\label{cond}
H^*=H,\quad \det H\neq 0, \quad\mbox { and  }\quad (HC)^T=HC,
\end{equation}
respectively.

If one chooses another basis $\{\tilde  e_1,\ldots,\tilde e_n\}$, letting $\widetilde H$ and $\widetilde C$ be the matrices representing the form $\ell$ and the operator $A$ in this new basis and letting $M=(M_{i,j})$ be the transition matrix from the new basis to the old one, 
(i.e.,
$e_j=\sum_{i=1}^n M_{i,j} \tilde e_i,$
) 
then 
\begin{equation}
\label{transformation}
\widetilde H=(M^{-1})^* H M^{-1}  \quad\mbox{ and }\quad \widetilde C= M C \overline M^{-1}.
\end{equation}
Our goal is  to find a basis in which the matrix representation of the form $\ell$ and operator $A$ has a particularly simple form. In other words, if we define an action of the matrix group $GL_n(\mathbb C)$ on the pairs $(H, C)$ of $n\times n$ matrices satisfying \eqref{cond} by the mapping 
\begin{equation}
\label{action}
\bigl(M, (H, C)\bigr)\mapsto \Bigl(\left(M^{-1}\right)^* H M^{-1},  M C \overline M^{-1}\Bigr),\quad M\in GL_n(\mathbb C),
\end{equation}
then our goal is to choose a representative in each orbit of this action in a canonical way. This canonical representative is usually called the \emph{canonical}  or \emph{normal} form of the pair $(\ell, A)$.

 We let  $T_k$ be the $k\times k$ matrix whose $(i,j)$ entry is 1 if $j-i=1$ and zero otherwise, let $S_k$ be the $k\times k$ matrix whose $(i,j)$ entry is 1 if $j+i=k+1$ and zero otherwise,  let $I_k$ be the rank $k$ identity matrix, and let $J_{\lambda,k}=\lambda I_k+T_k$ be the standard $k\times k$ Jordan block corresponding to the eigenvalue $\lambda$.

To succinctly define new matrices constructed from others, we write
\[
M_1\oplus M_2\oplus\cdots\oplus M_{k}=\bigoplus_{i=1}^k M_i
\]
to denote the block diagonal matrix whose diagonal entries are the matrices $M_1,\ldots, M_k$. For $\lambda\in\C$, we define the $k\times k$ or $2k\times 2k$  matrix $C_{\lambda,k}$ by
\[
C_{\lambda,k}:= 
\begin{cases}
J_{\lambda,k}&\mbox{ if } \lambda\in\R\\
\left(
\begin{array}{cc}
0&  J_{\lambda^2,k}\\
I_k& 0
\end{array}
\right)
&\mbox{ otherwise},
\end{cases}
\]
where  $0$ denotes a matrix of appropriate size with zero in all entries. 
We define corresponding matrices $H_{\lambda, k}$ by
\[
H_{\lambda,k}:= 
\begin{cases}
S_{k}&\mbox{ if } \lambda\in\R\\
S_{2k}&\mbox{ otherwise}.
\end{cases}
\]

 For a nonnegative integer $k$, we define
\begin{equation}
\label{filt1}
W_\lambda^{(k)}:=\mathrm{span}_{\C}\left\{v\in W\,:\,(A^2-\lambda^2 I)^kv=0\mbox{ or }\left(A^2-\overline{\lambda^2} I\right)^kv=0\right\}.
\end{equation}
Since $A^2$ is linear,
we can enumerate its eigenvalues, letting $\lambda_1^2,\ldots,\lambda_\mu^2$ be the real  eigenvalues of $A^2$  and $\lambda_{\mu+1}^2,\ldots,\lambda_\gamma^2$ be the distinct eigenvalues of $A^2$ with positive imaginary part. In the canonical forms below, we assume that each $\lambda_i$ is the principle square root of $\lambda_i^2$. 
Since the linear operator $A^2$ is $\ell$-self-adjoint, it is easy to show
 (see, for example,  \citep[Theorem 4.2.4]{gohberg2006indefinite}) that the  space $W$ can be decomposed into pairwise-$\ell$-orthogonal $A^2$-invariant subspaces
\begin{align}\label{Cn decomposition}
W=W_{\lambda_1}^{(n)}\oplus W_{\lambda_2}^{(n)}\oplus\cdots \oplus W_{\lambda_\gamma}^{(n)}.
\end{align}
\begin{remark}\label{Gohberg--Lancaster--Rodman}
In \cite{gohberg2006indefinite}, the authors refine this decomposition of $W$, obtaining a canonical form for $(\ell, A^2)$. For an $\ell$-self-adjoint linear operator $B$, their theorem, \cite[Theorem 5.1.1]{gohberg2006indefinite}, states that the domain of $B$ can be decomposed into $B$-invariant, pairwise $\ell$-orthogonal subspaces such that there exists a basis with respect to which the restrictions of $\ell$ and $B$ to the decomposition's component subspaces are represented by matrices of the form $\pm S_{k}$ and $J_{\eta,k}$ if $\eta\in \R$ or $\pm S_{2k}$ and $J_{\eta,k}\oplus J_{\overline{\eta},k}$ if $\eta\not\in \R$ (this gives a canonical form for  $(\ell, A^2)$ by letting $B=A^2$).
\end{remark}
Note that  $W_{\lambda_i}^{(n)}$ is also $A$-invariant. Indeed , if $v\in W_{\lambda_i}^{(n)}$ and $(A^2-\lambda^2 I)^kv=0$, then 
\[
\left(A^2-\overline{\lambda_i^2}I\right)^n (Av)=A(A^2-\lambda_i^2I)^n v=0,
\]
which shows that $Av\in W_{\lambda_i}^{(n)}$. Similarly, if $v\in W_{\lambda_i}^{(n)}$ and $(A^2-\overline {\lambda^2} I)^kv=0$, then 
\[
\left(A^2-\lambda_i^2I\right)^n (Av)=A(A^2-\overline{\lambda_i^2}I)^n v=0,
\]
which shows that $Av\in W_{\lambda_i}^{(n)}$. This completes the proof of $A$-invariancy of $W_{\lambda_i}^{(n)}$.

Accordingly, we can normalize $\ell$ and $A$ on the spaces $W_{\lambda_i}$ separately to obtain a general canonical form. 


\begin{theorem}\label{simultaneous canonical form theorem}
The domain of an $\ell$-self-adjoint antilinear operator $A$ can be decomposed into $A$-invariant, pairwise $\ell$-orthogonal subspaces such that there exists a basis with respect to which the restrictions of $\ell$ and $A$ to the decomposition's component subspaces are represented by matrices of the form $\pm H_{\lambda,k}$ and $C_{\lambda,k}$ where $\lambda\in\{\lambda_1,\lambda_2,\ldots,\lambda_\gamma\}$ and $k\in\N$. The corresponding block  diagonal matrices representing $\ell$ and $A$ are unique up to a permutation of the blocks on the diagonal.
\end{theorem}
\begin{proof}
Since the decomposition in \eqref{Cn decomposition} is pairwise $\ell$-orthogonal and $A$-invariant, the result is a corollary of Propositions \ref{real e-val canonical form prop}, \ref{zero e-val canonical form prop}, \ref{negative e-val canonical form prop}, and \ref{nonreal e-val canonical form prop}.
\end{proof}

\begin{remark}\label{Hong--Horn}
In \cite[Theorem 3.1]{hong1988canonical}, the authors show that an antilinear operator $A$ can be represented by a matrix in the form of the matrix given in Theorem \ref{simultaneous canonical form theorem} representing the antilinear operator, that is, the domain of  $A$ can be decomposed into $A$-invariant subspaces on which $A$ is represented by $C_{\lambda,k}$ where $\lambda\in\{\lambda_1,\lambda_2,\ldots,\lambda_\gamma\}$ and $k\in\N$ (note, this is achieved without the assumption that $A$ is $\ell$-self-adjoint for some Hermitian form $\ell$). 
\end{remark}
A canonical form for a nonsingular antilinear operator is fully determined by the Jordan matrix representing its square, and we have a similar relationship between Theorem \ref{simultaneous canonical form theorem} and the Gohberg--Lancaster--Rodman form, recorded in the following lemma.
\begin{lemma}\label{GLR invariant lemma}
If $A$ is nonsingular then the canonical form for $(\ell,A)$ given in Theorem \ref{simultaneous canonical form theorem} is determined by the Gohberg--Lancaster--Rodman form for $(\ell,A^2)$.
\end{lemma}

\section{Relation to CR Geometry}\label{applications}

In this section we demonstrate how the considered pairs of algebraic objects appear naturally in the study of a certain class of CR manifolds of hypersurface type.
CR manifolds of hypersurface type are real hypersurfaces in a complex space $\C^{n+1}$ and originally were introduced in order to study the biholomorpic equivalence between  domains in $\C^{n+1}$ via their boundaries (see, for example, the monograph \cite{jacobowitz}). The complex structure of $\mathbb C^{n+1}$ induces additional nontrivial structures on a hypersurface $M$. Namely, for every $x\in M$,  let $D_x$ be the maximal complex subspace of the tangent space $T_xM$, $D_x=T_xM\cap i T_xM$, where $i T_xM$ is the real subspace in $\mathbb R^{2n+2}$ ($\cong C^{n+1}$) obtained from $T_xM$ via the multiplication by $i$. The collection of hyperplanes $D=\{D_x\}_{x\in M} $ defines a corank 1 subbundle of $TM$, that is, a corank $1$ distribution on $M$. 

By construction, the multiplication by $i$ restricted to $D(x)$ defines and endomorphism of $D(x)$ that will be denoted by $J_x$. By construction, $J_x^2=I_x$, where $I_x$ is the identity operator on $D_x$.  The operator $J_x$ extends linearly to $\C D_x=D\otimes \C \subset \C T_xM= T_xM\otimes \C$, and, since $J_x^2=-I$, $ \C D$ splits into a direct sum of the $i$-eigenspace and $-i$-eigenspace of $J_x$, denoted by $E_x$ and $\overline E_x$, respectively. The collections $E=\{E_x\}_{x\in M}$ and $\overline E=\{\overline E_x\}_{x\in M}$ define subbundles of the complexified tangent bundle $\C T_xM$ and both of these subbundles are involutive, that is, $[E, E]\subset E$ and $[\overline E, \overline E]\subset \overline E$, where, for example,  by $[E, E]_x$  we mean the linear span of Lie brackets, evaluated at $x$, of any two sections of the bundle $E$. This involutivity comes from the fact that if $(z_1, \ldots z_{n+1})$ are standard coordinates in $\C^{n+1}$ then at every point  
$$E=\mathbb CD\cap\mathrm{span}\left(\left\{\frac{\partial}{\partial z_i} \right\}_{i=1}^{n+1}\right), \quad \overline E=\mathbb CD\cap\mathrm{span}\left(\left\{\frac{\partial}{\partial \bar z_i}\right\}_{i=1}^{n+1}\right).$$	

Keeping all of this in mind, an \emph{abstract CR structure of hypersurface type} is a triple $(M, D, J)$, where $M$ is an odd dimensional real manifold, $D$ is a corank 1 distribution, and $J:D\to D$ is an operator that preserves each fiber of $D$ (i.e., $J D_{x}=D_{x}\quad\forall x\in M$) such that $J$ is linear on each fiber, $J^2=-I$, and the corresponding subbundles $E$ and $\bar E$ of the complexified bundle $\C D$ are involutive. 

%

Now, given a fiber bundle $\pi:P\to M$,  let $\Gamma(P)$ be the set of smooth sections of $P$.
For $x\in M$, the Hermitian form $\mathcal{L}_x:E_x\to(\C TM)_p/(\C D)_p\cong \C$ known as the  Levi form of $(M,D,J)$ is defined, up to  a real multiple,  as 
\[
\mathcal{L}_x(X_x,Y_x) :=\frac{1}{2i}\left[X,\overline{Y}\right]_x\quad\quad\forall X,Y\in\Gamma(E).
\] 
The kernel $K_x$ of this Hermitian form is called the \emph{Levi kernel} of the CR structure at the point $x$. The CR structure  $(M, D, J)$ is called Levi-nondegenerate if $K_x=0$ at every point. The local differential geometry of Levi-nondegenerate CR structures is well understood (see \cite{cartan, chern, tanakaCR}). In recent years, interest arose in the uniformly degenerate structures, that is,  when $K_x\neq 0$ for every $x$; see \cite{porter2017absolute} for the list of references.

Assume now that the Levi kernel $K_x$ is one-dimensional for every $x$. First, the degenerate Hermitian form $\mathcal{L}_x$ on $E_x$ factors through $E_x/K_x$, defining the nondegenerate Hermitian form $\ell_x$ on $E_x/K_x$, that is, $\ell_x$ is well defined by $\ell_x( \pi v,\pi w):=\mathcal{L}_x( v,  w)$ with $\pi:H_x\to H_x/K_x$ denoting the canonical projection projection. Second, for  $x\in M$ and $v\in K_x$, we define the antilinear operator $A_{x}: H_x/K_x\to H_x/K_x$ by choosing $V\in\Gamma(K_x)$ such that $V(x)=v$ and
\[
A_{x}(Y_x)=\left[V,\overline{Y}\right]_x \pmod{K_x\oplus \overline{H}_x}\quad\quad\forall Y\in \Gamma(H_x/K_x).
\]
The  antilinear operator $A_x$ is defined, up to a complex multiple and is $\ell_x$-self-adjoint (see \cite{porter2017absolute}).

It turns out that a pair $(\mathbb R \ell_x, \C A_x)$ is a basic invariant of the CR structure under consideration at a point $x$ and our Theorem \ref{simultaneous canonical form theorem} gives the classification of pairs $(\ell_x, A_x)$ and therefore of the these basic invariants. In 
\cite{porter2017absolute}, the structure of an absolute parallelism  (i.e., a canonical frame/coframe in a certain bundle over $M$) and maximally symmetric models were found in the particular case where $A_x^3$ is a scalar multiple of $A_x$ at every point. The reason for this somewhat weird condition is that only in this case was it possible to apply a certain version of the machinery of the prolongation of filtered structures for the construction of an absolute parallelism. In many cases where $A_x^3$ is not a scalar multiple of $A_x$ it is not even clear if there exist CR-structures on which its symmetry group acts transitively, and the existing differential geometric methods only offer hope to obtain upper bounds for the dimension of the symmetry group of the most symmetric models with pairs $(\mathbb R \ell_x,\C A_x)$ lying in the same prescribed orbit under the natural $GL$-action for every $x$. 
In any case, we expect that the classification given by our Theorem \ref{simultaneous canonical form theorem} will be useful for finding an upper bound for dimensions of  symmetry groups  and other questions related to CR structures with one-dimensional Levi kernel for which $A_x^3$ is not a scalar multiple of $A_x$.


\section{Normal Forms for Restrictions to Generalized Eigenspaces}\label{antilinear operators special cases}
In this section we  obtain a canonical form for the restrictions of $\ell$ and $A$ to the spaces $W_\lambda^{(n)}$, and these results can be taken together to obtain the canonical form in Theorem \ref{simultaneous canonical form theorem}. The approach we employ varies depending on the eigenvalue $\lambda^2$ of $A^2$, so this section is structured with subsections, each dedicated to a case where $\lambda^2$ belongs to a different family. We repeatedly use the following lemma, which is completely analogous to a standard property of linear self-adjoint operators. 
\begin{lemma}\label{ell-orthogonal complement}
If $V\subset \C^n$ is an $A$-invariant subspace on which $\ell$ is nondegenerate then the $\ell$-orthogonal complement $V^{\perp_\ell}$ of $V$ is also $A$-invariant.
\end{lemma}
\begin{proof}
Since $V$ is $A$-invariant,  for any $w\in V$ , we have that $Aw\in V$, which implies that, for $v\in V^{\perp_\ell}$, we have $\ell(Aw, v)=0$. Therefore,  since $A$ is $\ell$-self-adjoint, for $v\in V^{\perp_\ell} $ and $w\in V$, we have $\ell(Av, w)=\ell(Aw, v)=0$, which implies that $Av\in V^{\perp_\ell}$.
\end{proof}

\subsection{Treating Generalized Eigenspaces with Positive Eigenvalues}\label{positive eigenvalue subsection}
Throughout this subsection we assume $\lambda^2>0$, and this subsection's main result is  Proposition \ref{real e-val canonical form prop}.

For this special case with $\lambda^2>0$, we define three additional filtrations of $W_\lambda^{(n)}$. Namely,
\begin{equation}
\label{filt2}
W_\lambda^{(k)\pm}:=\left\{x\in W\,:\,(A\mp\lambda I)(A^2-\lambda^2 I)^{k-1}x=0\right\},
\end{equation}
and 
\begin{equation}
\label{filt3}
\widetilde{W}_\lambda^{(k)}:=\left\{x\in W: \left(A-\lambda I\right)^kx=0\right\}.
\end{equation}

The following two lemmas address the relationship between the filtrations $\left\{W_\lambda^{(k)}\right\}$, $\left\{W_\lambda^{(k)\pm}\right\}$, and $\left\{\widetilde{W}_\lambda^{(k)}\right\}$, defined by \eqref{filt1}, \eqref{filt2}, and \eqref{filt3}, respectively. Note that, for each $k$, $W_\lambda^{(k)\pm}$ and $\widetilde{W}_\lambda^{(k)}$ are vector spaces over $\mathbb R$ but not over $\mathbb C$.  In principle, these lemmas can be deduced from the Hong--Horn canonical forms for antilinear operators from \cite[Theorem 3.1]{hong1988canonical} (see also Remark \ref{Hong--Horn} above), but we prefer to give an independent geometric proof of these Lemmas, first, in order to make the presentation self-contained (as the source \cite{hongthesis}, where  \cite[Theorem 3.1]{hong1988canonical} is proved, is not easily available), second, because our proofs of these Lemmas are the main ingredient in the new geometric proof of Hong and Horn's result (outlined in section \ref{A canonical form for antilinear operators}), and, third, because this proof seems to be interesting by itself.

\begin{lemma} \label{Splitting A^2 filtration} For all positive integers $k$,  we have $W_\lambda^{(k)}/W_\lambda^{(k-1)}=W_\lambda^{(k)+}/W_\lambda^{(k-1)}\oplus W_\lambda^{(k)-}/W_\lambda^{(k-1)}$. Moreover, $W_\lambda^{(k)}/W_\lambda^{(k-1)}=\text{span}_\C\left(W_\lambda^{(k)+}/W_\lambda^{(k-1)}\right)$.
\end{lemma}

\begin{proof} If $x\in W_\lambda^{(k)}$ then ${\lambda}x\pm Ax\in W_\lambda^{(k)\pm}$ because 
\[
(A\mp {\lambda}I)(A^2-\lambda^2 I)^{k-1}( {\lambda}x\pm Ax)=(A^2-\lambda^2 I)^{k}x=0,
\]
which shows
\[
\left\{ {\lambda}x\pm Ax\,:\, x\in W_\lambda^{(k)}\right\}/W_\lambda^{(k-1)} \subset W_\lambda^{(k)\pm}/W_\lambda^{(k-1)}.
\]
Accordingly,
\begin{align}\label{splitting lemma inclusion eqn}
W_\lambda^{(k)}/W_\lambda^{(k-1)} & \stackrel{*}{=} \left\{ {\lambda}x- Ax\,:\, x\in W_\lambda^{(k)}\right\}/W_\lambda^{(k-1)}\oplus\left\{ {\lambda}x+ Ax\,:\, x\in W_\lambda^{(k)}\right\}/W_\lambda^{(k-1)}\nonumber\\
&\subset W_\lambda^{(k)-}/W_\lambda^{(k-1)}\oplus W_\lambda^{(k)+}/W_\lambda^{(k-1)}\nonumber\\
&\stackrel{**}{\subset} W_\lambda^{(k)}/W_\lambda^{(k-1)},
\end{align}
where  ** holds because $A^2-\lambda^2 I=(A+ {\lambda}I)(A- {\lambda}I)$ and * holds for the following reason. Both $\left\{ {\lambda}x- Ax\,:\, x\in W_\lambda^{(k)}\right\}/W_\lambda^{(k-1)}$ and $\left\{ {\lambda}x+ Ax\,:\, x\in W_\lambda^{(k)}\right\}/W_\lambda^{(k-1)}$ are disjoint subsets of $W_\lambda^{(k)}/W_\lambda^{(k-1)} $ because they belong to the kernel of $A+\lambda I:W_\lambda^{(k)}/W_\lambda^{(k-1)}\to W_\lambda^{(k)}/W_\lambda^{(k-1)}$ and $A-\lambda I:W_\lambda^{(k)}/W_\lambda^{(k-1)}\to W_\lambda^{(k)}/W_\lambda^{(k-1)}$ respectively, and these kernels are disjoint because if $v$ is in both kernels then $\lambda v\equiv -\lambda v\pmod{W_\lambda^{(k-1)}}$. This shows that the direct sum on the right side of * is naturally a subset of $W_\lambda^{(k)}/W_\lambda^{(k-1)} $. On the other hand, for any $v\in W_\lambda^{(k)}/W_\lambda^{(k-1)} $, we have
\[
v=\left(\lambda\left( \frac{v}{\lambda}\right)-A\left( \frac{v}{\lambda}\right)\right)+\left(\lambda\left( \frac{v}{\lambda}\right)+A\left( \frac{v}{\lambda}\right)\right)
\]
which shows that $W_\lambda^{(k)}/W_\lambda^{(k-1)} $ is contained in the direct sum on the right side of *.

By \eqref{splitting lemma inclusion eqn},  $W_\lambda^{(k)}/W_\lambda^{(k-1)}=\text{span}_\C\left(W_\lambda^{(k)+}/W_\lambda^{(k-1)}\right)$ because $W_\lambda^{(k)+}/W_\lambda^{(k-1)}=iW_\lambda^{(k)-}/W_\lambda^{(k-1)}$.
\end{proof}
\begin{remark}
Notice, we have already used the special condition $\lambda^2>0$ of \ref{positive eigenvalue subsection}, because Lemma \ref{Splitting A^2 filtration} relies on the fact that $A^2-\lambda^2 I=(A+ {\lambda}I)(A- {\lambda}I)$.
\end{remark}

%
%
%

\begin{lemma}\label{half space lemma}
Any basis of the real vector space $\widetilde{W}_\lambda^{(k)}$ is also a basis of the complex vector space ${W}_\lambda^{(k)}$.
\end{lemma}
\begin{proof}
When $k=1$, the statement follows from Lemma \ref{Splitting A^2 filtration} because  $\widetilde{W}_\lambda^{(1)}={W}_\lambda^{(1)+}$ and ${W}_\lambda^{(0)}=0$. Proceeding by induction, let us assume any basis of the real vector space $\widetilde{W}_\lambda^{(k-1)}$  is also a basis of the complex vector space ${W}_\lambda^{(k-1)}$.  Suppose $\dim \widetilde{W}_\lambda^{(k-1)}=l$ and $\dim \widetilde{W}_\lambda^{(k)}/\widetilde{W}_\lambda^{(k-1)}=m$, and let $\{e_1,\ldots, e_{l+m}\}$ be a basis of $\widetilde{W}_\lambda^{(k)}$. Without loss of generality, we can assume $\{e_1,\ldots, e_{l}\}\subset \widetilde{W}_\lambda^{(k-1)}$ because this assumption does not change the real or complex span of $\{e_1,\ldots, e_{l+m}\}$.

First, we show that the vectors $e_{l+1},\ldots, e_{l+m}$ are linearly independent over $\C$ modulo ${W}_\lambda^{(k-1)}$. For this, consider a vector $v\in \text{span}_\C\{e_{l+1},\ldots, e_{l+m}\}$ with coefficients $\alpha_{l+1},\ldots,\alpha_{l+m}\in\R$ and $\beta_{l+1},\ldots,\beta_{l+m}\in\R$ such that 
\[
v:=\sum_{j=l+1}^{l+m}(\alpha_j+i\beta_j)e_i\in W_\lambda^{(k-1)}.
\]
Set
\[
v_+:=\sum_{j=l+1}^{l+m}\alpha_je_i\quad\mbox{ and }\quad v_-:=\sum_{j=l+1}^{l+m}i\beta_je_i.
\]
Since $\widetilde{W}_\lambda^{k}$ is a real vector space, $v_+\in \widetilde{W}_\lambda^{k}$, and hence, by \eqref{filt2},
\[
A(A-\lambda I)^{k-1}v_+=\lambda (A-\lambda I)^{k-1}v_+.
\]
Therefore
\begin{align}\label{half space lemma eqn1}
(A+\lambda I)(A-\lambda I)^{k-1}v_+=2\lambda(A-\lambda I)^{k-1}v_+
\end{align}
Notice $(A+\lambda I)^kv_-=0$ because, for all $l<j\leq l+m$, $(A+\lambda I)^k(i\beta_i e_j)=-i\beta_j(A-\lambda I)^ke_j=0$. Since $v\in W_\lambda^{(k-1)}$ and $(A+\lambda I)^kv_-=0$, 
\begin{align*}
0=(A+\lambda I)(A^2-\lambda^2 I)^{k-1}v
&=(A+\lambda I)(A^2-\lambda^2 I)^{k-1}v_+ + (A-\lambda I)^{k-1}(A+\lambda I)^{k}v_-\\
&=(A+\lambda I)(A^2-\lambda^2 I)^{k-1}v_+,
\end{align*}
and hence
\begin{align}\label{half space lemma eqn2}
(A^2-\lambda^2 I)^{k-1}v_+\in \ker (A+\lambda I).
\end{align}
Furthermore, $(A-\lambda I)^kv_+=0$ because, for all $l<j\leq l+m$, $(A-\lambda I)^k\alpha_je_j=0$, so
\begin{align}\label{half space lemma eqn3}
(A^2-\lambda^2 I)^{k-1}v_+\in \ker (A-\lambda I).
\end{align}
Yet, $\ker (A-\lambda I)\cap \ker (A+\lambda I)=0$, \eqref{half space lemma eqn2} and \eqref{half space lemma eqn3} imply
\begin{align}\label{half space lemma eqn4}
v_+\in \ker (A^2-\lambda^2 I)^{k-1},
\end{align}
and \eqref{half space lemma eqn1} implies
\begin{align}\label{half space lemma eqn5}
(A^2-\lambda^2 I)^{k-1}v_+=(A+\lambda I)^{k-1}(A-\lambda I)^{k-1}v_+=(2\lambda)^{k-1}(A-\lambda I)^{k-1}v_+.
\end{align}
Together, \eqref{half space lemma eqn4} and \eqref{half space lemma eqn5} imply that 
\[
v_+\in \ker (A-\lambda I)^{k-1}=\widetilde{W}_\lambda^{(k-1)}=\text{span}_\R\{e_1,\ldots,e_l\}
\]
and hence $v_+=0$ because $\text{span}_\R\{e_1,\ldots,e_l\}\cap\text{span}_\R\{e_{l+1},\ldots,e_{l+m}\}=0$. Note that $v_+=0$ implies $\alpha_{l+1}=\cdots=\alpha_{l+m}=0$ because $e_{l+1},\ldots,e_{l+m}$ are linearly independent over $\R$.
Repeating the same argument with $v$ replaced by $iv$ yields $v_-=0$ and $\beta_{l+1}=\cdots=\beta_{l+m}=0$ as well. Hence $v=0$, which shows that
\begin{align}\label{half space lemma eqn6}
\text{span}_\C\{e_{l+1},\ldots, e_{l+m}\}\cap W_\lambda^{(k-1)}=\text{span}_\C\{e_{l+1},\ldots, e_{l+m}\}\cap \text{span}_\C\{e_{1},\ldots, e_{l}\}=0.
\end{align}

Let us now establish the vector space isomorphism $\widetilde{W}_\lambda^{(k)}/\widetilde{W}_\lambda^{(k-1)}\cong {W}_\lambda^{(k)+}/{W}_\lambda^{(k-1)}$. The cosets 
\[
e_{l+1}+W_\lambda^{(k-1)},\ldots,e_{l+m}+W_\lambda^{(k-1)}
\]
are linearly independent vectors (over $\R$) in the space $ W_\lambda^{(k)+}/W_\lambda^{(k-1)}$. If we take an arbitrary vector $w+W_\lambda^{(k-1)}\in  W_\lambda^{(k)+}/W_\lambda^{(k-1)}$ then $(A-\lambda I)w\in W_\lambda^{(k-1)}$, so
\[
(A+\lambda I)w\equiv 2\lambda w\pmod{W_\lambda^{(k-1)}}.
\]
Hence,
\begin{equation}
\label{ww}
(2\lambda)^{1-k}(A+\lambda I)^{k-1}w\equiv w\pmod{W_\lambda^{(k-1)}}.
\end{equation}
Now observe that $(A+\lambda I)^{k-1}w\in \widetilde{W}_\lambda^{(k)}$. Indeed, from the definitions \eqref{filt1} and \eqref{filt2} and the fact that $w\in W_\lambda^{(k)+}$, it follows that 
$$(A-\lambda I)^k(A+\lambda I)^{k-1}w=(A-\lambda I)(A^2- \lambda^2I)^{k-1}w=0.$$
Hence, by \eqref{ww}, $w\in \widetilde{W}_\lambda^{(k)}$. Therefore, there exist real coefficients $a_{l+1},\ldots, a_{l+m}$ such that 
\[
w\equiv (2\lambda)^{1-k}(A+\lambda I)^{k-1}w\equiv \sum_{i=l+1}^{l+m}a_{i}e_i \pmod{ {W}_\lambda^{(k)}}.
\]
This shows that the cosets 
\[
e_{l+1}+W_\lambda^{(k-1)},\ldots,e_{l+m}+W_\lambda^{(k-1)}
\]
form a basis of $ W_\lambda^{(k)+}/W_\lambda^{(k-1)}$.
On the other hand, the cosets 
\[
e_{l+1}+\widetilde{W}_\lambda^{(k-1)},\ldots,e_{l+m}+\widetilde{W}_\lambda^{(k-1)}
\]
form a basis of $\widetilde{W}_\lambda^{(k)}/\widetilde{W}_\lambda^{(k-1)}$,
so the real vector spaces $\widetilde{W}_\lambda^{(k)}/\widetilde{W}_\lambda^{(k-1)}$ and $ W_\lambda^{(k)+}/W_\lambda^{(k-1)}$ are isomorphic.

Applying Lemma \ref{Splitting A^2 filtration}, we get
\[
\dim_\C W_\lambda^{(k)}/W_\lambda^{(k-1)}=\dim_\R W_\lambda^{(k)+}/W_\lambda^{(k-1)}=\dim_\R  W_\lambda^{(k)+}/W_\lambda^{(k-1)}=m,
\]
which implies
\[
\dim_\C W_\lambda^{(k)}=m+\dim W_\lambda^{(k-1)}=l+m.
\]
We have shown that $v\in W_\lambda^{(k-1)}$ implies $\alpha_{l+1}=\cdots=\alpha_m=\beta_{l+1}=\cdots=\beta_{l+m}=0$. 
In particular, we have shown that $v=0$ implies $\alpha_{l+1}=\cdots=\alpha_{l+m}=\beta_{l+1}=\cdots=\beta_{l+m}=0$. Therefore, $e_{l+1},\ldots, e_{l+m}$ are linearly independent over $\C$, that is,
\[
\dim_\C\text{span}_\C\{e_{l+1},\ldots, e_{l+m}\}=m,
\]
so, by the induction hypothesis and \eqref{half space lemma eqn6},
\[
\dim_\C\text{span}_\C\{e_{1},\ldots, e_{l+m}\}=l+m.
\]
Therefore, $W_\lambda^{(k)}=\text{span}_\C\{e_{1},\ldots, e_{l+m}\}$ because $e_{1},\ldots, e_{l+m}$ are $l+m$ linearly independent vectors (over $\C$) in $W_\lambda^{(k)}$.
\end{proof}

\begin{corollary}\label{decomposition corollary}
If $v\in {W}_\lambda^{(k)}$ then there exist unique vectors $v_{+},v_{-}\in\widetilde{W}_\lambda^{(k)}$ such that $v=v_++iv_-$.
\end{corollary}

%
%
%
%

Define $s_1$ to be the minimal natural number such that $W_\lambda ^{(s_1)}=W_\lambda^{(n)}$. We would like to find a vector $v\in \widetilde{W}_\lambda^{(s_1)}$ such that the space
\begin{align}\label{A inv nondeg space}
V=\text{span}_{\C} \left\{v, (A-\lambda I)v,\ldots,(A-\lambda I)^{s_1-1}v\right\}
\end{align}
is an $s_1$-dimensional $A$-invariant space on which $\ell$ is nondegenerate because we can then normalize $A$ and $\ell$ on the space $V$ and on the $\ell$-orthogonal complement of $V$ separately.  Proceeding throughout this subsection, for $v\in W_\lambda^{(n)}$, we adopt the notation of letting $v_+,v_-\in\widetilde{W}_\lambda^{(n)}$ be the unique vectors such that $v=v_++iv_-$, as given in Corollary \ref{decomposition corollary}.
\begin{lemma}\label{zero diagonals in H lemma}
If $H$ is a Hermitian $k\times k$ matrix, $\lambda>0$, and  $H J_{\lambda,k}$ is symmetric, then $H$ is a Hankel matrix satisfying 
\begin{align}\label{the zero diagonals condition}
H_{i,j}=0\quad\quad\forall i+j\leq k.
\end{align}
\end{lemma}
\begin{proof}
Let $H^{\prime}$ be the upper left $(k-1)\times (k-1)$ block of $H$. Symmetry of $H J_{\lambda,k}$ implies that $H^{\prime}J_{\lambda,k-1}$ is symmetric. Since the Lemma is vacuously true for $k=1$, we can proceed by induction, and assume $H^\prime$ is a Hankel matrix satisfying
\[
H_{i,j}^\prime=0\quad\quad\forall i+j\leq k-1.
\]
Computing the $(1,k)$ and $(k,1)$ entries of $H J_{\lambda,k}$ yields
\[
\Bigl(H J_{\lambda,k}\Bigr)_{1,k}=\lambda H_{1,k}+ H_{1,k-1} \quad\mbox{ and }\quad \Bigl(H J_{\lambda,k}\Bigr)_{k,1}=\lambda \overline{H_{1,k}}.
\]
Symmetry of $H J_{\lambda,k}$ allows us to equate the terms, so
\[
H_{1,k-1}=\lambda\left( \overline{H_{1,k}}-H_{1,k}\right) \in\{iz\,|\,z\in\R\}.
\]
Yet, since $H^\prime$ is both Hankel and Hermitian, its entries are all real numbers. In particular, $H_{1,k-1}\in \R$, so 
\[
H_{1,k-1}=0\quad\mbox{ and }\quad H_{1,k}=H_{k,1}\in\R.
\]
Equating $\Bigl(H J_{\lambda,k}\Bigr)_{2,k}$ with $\Bigl(H J_{\lambda,k}\Bigr)_{k,2}$ yields
\[
H_{1,k}-H_{2,k-1}=\lambda\left(H_{2,k}-\overline{H_{2,k}}\right)\in\{iz\,|\,z\in\R\}
\]
which implies $H_{k,1}=H_{2,k-1}$ because, by the induction hypothesis, $H_{2,k-1}\in\R$. Accordingly,  
\[
H_{i,j}=H_{1,k}\quad\quad \forall i+j=k+1
\]
because $H^\prime$ is Hankel.

We conclude this proof with induction. Supposing, for some $1< m\leq k$, we have $H_{k,j}=H_{j,k}$ and $H_{k,j}=H_{j+1,k-1}$ for all $j<m$, let us establish that $H_{k,m}=H_{m,k}$ and $H_{k,m}=H_{m+1,k-1}$, where we interpret $H_{k,m}=H_{m+1,k-1}$ as vacuously true for $m= k$ (i.e., since $H_{i,j}$ is only defined for $\max\{i,j\}\leq k$, we can extend the definition of $H_{i,j}$ for $\max\{i,j\}> k$ in a way that satisfies the equations $H_{k,k+j}=H_{k+j,k}$ and $H_{k,k+j}=H_{k+j+1,k-1}$ for all $j$ by construction, and, of course, this extension's definition has no relevance to the normalization of $H$). Symmetry of $H J_{\lambda,k}$ implies
\[
H_{m,k-1}+\lambda H_{m,k}=\Bigl(H J_{\lambda,k}\Bigr)_{m,k}=\Bigl(H J_{\lambda,k}\Bigr)_{k,m}=H_{k,m-1}+ \lambda H_{k,m},
\]
and hence
\[
H_{m,k}-H_{k,m}=\lambda^{-1}(H_{k,m-1}-H_{m,k-1})=0
\]
because, by the induction hypothesis, $H_{k,m-1}=H_{m,k-1}$. If $m=k$ then there is nothing more to check, that is, $H_{k,m}=H_{m+1,k-1}$ is vacuously true. Similarly, if $m=k-1$ then we have already shown $H_{k,m}=H_{m+1,k-1}$. For $m<k-1$, we have
\[
H_{m+1,k-1}+\lambda H_{m+1,k}=\Bigl(H J_{\lambda,k}\Bigr)_{m+1,k}=\Bigl(H J_{\lambda,k}\Bigr)_{k,m+1}=H_{k,m}+ \lambda H_{k,m+1},
\]
and hence
\[
H_{k,m}-H_{m+1,k-1}=\lambda( H_{m+1,k}- H_{k,m+1})=\lambda( H_{m+1,k}- \overline{H_{m+1,k}})\in\{z\,|\, iz\in\R\},
\]
which implies $H_{k,m}=H_{m+1,k-1}$ because, since $H^\prime$ is a real matrix, $H_{m+1,k-1}\in \R$. This completes the proof by induction.
\end{proof}
\begin{lemma}\label{normalizing ell for lambda>0}
If a nondegenerate Hermitian form $\ell$ and antilinear operator $A$ are represented respectively by the $k\times k$ matrices $H$ and $J_{\lambda,k}$, where $H$ is a Hankel matrix satisfying 
\[
H_{i,j}=0\quad\quad\forall\, i+j\leq k,
\]
then there is a basis with respect to which $\ell$ and $A$ are represented by $\pm S_k$ and $J_{\lambda,k}$ respectively.
\end{lemma}
\begin{proof}
Every transformation of the matrices representing $\ell$ and $A$ given by the rule \eqref{transformation} can be induced by a change of basis, so it will suffice to find $M$ such that 
\begin{align}\label{tranformation rule normalizing H}
M^* H M=S_k\quad \mbox{ and }\quad M^{-1}J_{\lambda, k}\overline{M}=J_{\lambda, k}.
\end{align}
To satisfy $M^{-1}J_{\lambda, k}\overline{M}=J_{\lambda, k}$, let us suppose $M$ is a real upper-triangular Toeplitz matrix, and define $h_0,\ldots, h_{k-1}\in \R$ and $\alpha_0,\ldots, \alpha_{k-1} \in \R$ to be the coefficients for which 
\[
H=S_{k}\left(\sum_{i=1}^{k}h_{i-1}T_{k}^{i-1}\right)\quad\mbox{ and }\quad M=\sum_{i=1}^{k}\alpha_{i-1}T_k^{i-1}.
\]
Note, $h_1,\ldots, h_k$ must be real because $H$ is Hermitian and Hankel. For our particular choice of $M$, we have $M^*=S_k M S_k$, so
\[
M^* H M=S_{k}\left(\sum_{i=1}^k\alpha_{i-1}T_k^{i-1}\right)\left(\sum_{i=1}^{k}h_{i-1}T_{k}^{i-1}\right)\left(\sum_{i=1}^k\alpha_{i-1}T_k^{i-1}\right)=S_k\left(\sum_{i=0}^{k-1} \sum_{r+s+t=i}\alpha_r\alpha_s h_t T_{k}^i\right).
\]
Therefore, we need to solve the equation
\begin{align}\label{Toeplitz coefficients normalizing H}
\sum_{i=0}^{k-1} \sum_{r+s+t=i}\alpha_r\alpha_s h_t T_{k}^i=\pm I_k,
\end{align}
that is, we need to choose $\alpha_i$ such that \eqref{Toeplitz coefficients normalizing H} holds. Comparing entries of the main diagonal in \eqref{Toeplitz coefficients normalizing H}, we find that $\alpha_0= h_0^{-1/2}$, so let us choose $\alpha_0=|h_0|^{-1/2}$. Note, $h_0\neq 0$ because $\ell$ is nondegenerate, and hence this choice of $\alpha_0$ is well defined.  Having fixed $\alpha_0$, comparing entries in the first super-diagonal of \eqref{Toeplitz coefficients normalizing H} shows that we can choose $\alpha_1$ as the solution to a linear equation with real coefficients so that entries in the first super-diagonal of \eqref{Toeplitz coefficients normalizing H} match. Proceeding similarly, for $1<j< k$, after choosing $\alpha_0,\ldots, \alpha_{j-1}$ so that entries in the main diagonal and the first $j-1$ super-diagonals of \eqref{Toeplitz coefficients normalizing H} match, comparing entries in the $j$ super-diagonal of \eqref{Toeplitz coefficients normalizing H} shows that we can choose $\alpha_{j}$ as the solution to a linear equation with real coefficients so that entries in the $j$ super-diagonal of \eqref{Toeplitz coefficients normalizing H} match; moreover, the variable $\alpha_j$ does not appear in the first $j-1$ super-diagonals of \eqref{Toeplitz coefficients normalizing H}, so, by choosing $\alpha_j$ in this way, we ensure that entries the first $j$ super-diagonals of \eqref{Toeplitz coefficients normalizing H} match. By choosing $\alpha_0,\ldots, \alpha_k$ in this way we obtain \eqref{tranformation rule normalizing H} by construction.
\end{proof}

\begin{lemma}\label{reducing the general positive eigenvalue case}
There exists a vector $v\in W_\lambda^{(n)}$ such that the space in \eqref{A inv nondeg space} is an $s_1$-dimensional $A$-invariant space on which $\ell$ is nondegenerate.
\end{lemma}
\begin{proof}

It can be seen from the Gohberg--Lancaster--Rodman canonical form for $\ell$ and $A^2$ (given in \cite[Theorem 5.1.1]{gohberg2006indefinite} and summarized in Remark \ref{Gohberg--Lancaster--Rodman}) that there exists a vector $v^\prime\in W_\lambda^{(s_1)}$ for which 
\begin{align}\label{GLR nonzero pairing}
\ell\left(v^\prime,(A^2-\lambda^2I)^{(s_1-1)}v^\prime\right)\neq 0.
\end{align}
Using the decomposition of Corollary \ref{decomposition corollary}, define the coefficients
\[
a_0:=\ell\bigl(v^\prime_+,(A-\lambda I)^{(s_1-1)}v^\prime_+\bigr)+\ell\bigl(v^\prime_-,(A-\lambda I)^{(s_1-1)}v^\prime_-\bigr),
\]
\[
a_1:=\ell\bigl(v^\prime_+,(A-\lambda I)^{(s_1-1)}v^\prime_+\bigr)-\ell\bigl(v^\prime_-,(A-\lambda I)^{(s_1-1)}v^\prime_-\bigr),
\]
and  
\[
b_1:=\ell\bigl(v^\prime_-,(A-\lambda I)^{(s_1-1)}v^\prime_+\bigr)-\ell\bigl(v^\prime_+,(A-\lambda I)^{(s_1-1)}v^\prime_-\bigr).
\]
By direct computation, we obtain the finite Fourier series
\begin{align} \label{nonzero pairing expanded a}
2(2\lambda)^{1-s_1}\ell\left(\left(e^{i\theta}v^\prime\right)_{+},(A-\lambda I)^{s_1-1}\left(e^{i\theta}v^\prime\right)_{+}\right)&=a_0+a_1\cos(2\theta)+ b_1\sin(2\theta)
\end{align}
Also, since $v^\prime_+,v^\prime_-\in \widetilde{W}_\lambda^{(s_1)}$, $A(A-\lambda I)^{s_1-1}v^\prime_+=\lambda(A-\lambda I)^{s_1-1}v^\prime_+$ and $A(A+\lambda I)^{s_1-1}iv^\prime_- =-\lambda (A+\lambda I)^{s_1-1}iv^\prime_-$, and hence
\begin{align*}
(A^2-\lambda^2I)^{(s_1-1)}v^\prime &=(A+\lambda I)^{(s_1-1)}(A-\lambda I)^{(s_1-1)}v^\prime_++(A-\lambda I)^{(s_1-1)}(A+\lambda I)^{(s_1-1)}iv^\prime_-\\
&=(2\lambda)^{s_1-1}(A-\lambda I)^{(s_1-1)}v^\prime_++i(2\lambda)^{s_1-1}(A-\lambda I)^{(s_1-1)}v^\prime_-.
\end{align*}
So, by \eqref{GLR nonzero pairing},
\begin{align}\label{nonzero pairing expanded b}
0 &\neq (2\lambda)^{1-s_1}\ell\left(v^\prime,(A^2-\lambda^2I)^{(s_1-1)}v^\prime\right) =a_0+ib_1.
\end{align}

If the left side of \eqref{nonzero pairing expanded a} is zero for all $\theta\in \R$ then $a_0=a_1=b_1=0$, so, by \eqref{GLR nonzero pairing}, there exists $\theta\in\R$ such that 
\begin{align}\label{nonzero pairing}
 \ell\left(\left(e^{i\theta}v^\prime\right)_{+},(A-\lambda I)^{s_1-1}\left(e^{i\theta}v^\prime\right)_{+}\right)\neq 0.
 \end{align}


Fixing $\theta\in \R$ so that \eqref{nonzero pairing} holds, define
\begin{align}\label{top v for v-chain}
v:=\left(e^{i\theta}v^\prime\right)_+,
\end{align}
so, by  \eqref{nonzero pairing},
\begin{align}\label{new nonzero pairing}
\ell\bigl(v,(A-\lambda I)^{s_1-1}v\bigr)\neq0.
\end{align}
Proceeding, let $V$ be as in \eqref{A inv nondeg space} with $v$ as in \eqref{top v for v-chain}. Define basis vectors 
\[
e_i=(A-\lambda I)^{i-1}v\quad\quad\quad (i=1,\ldots,s_1).
\]
The matrix representing the restriction $A|_V$ of $A$ to $V$ with respect to the basis $\{e_i\}_{1\leq i\leq s_1}$ is $J_{\lambda,s_1}$. Let $H$ be the matrix representing the restriction of $\ell$ to $V$ with respect to the basis $\{e_i\}_{1\leq i\leq s_1}$. Since $A$ is $\ell$-self-adjoint, $H J_{\lambda,s_1}$ is symmetric. Therefore, applying Lemma \ref{zero diagonals in H lemma}, $H$ is a Hankel matrix satisfying 
\[
H_{i,j}=0\quad\quad\forall i+j\leq k,
\]
and hence, by \eqref{new nonzero pairing},
\[
\det H=\sqrt{2}\sin\left(\frac{2k\pi+\pi}{4}\right) H_{1,k}^k=\sqrt{2}\sin\left(\frac{2k\pi+\pi}{4}\right)\Bigl(\ell\left(v, (A-\lambda I)^{s_1-1}v\right)\Bigr)^k\neq0.
\]
That is, $\ell$ is nondegenerate on $V$, as was needed.
\end{proof}

\begin{corollary}\label{C and ell normalized dim 1 eSpace}
There is an $s_1$-dimensional $A$-invariant space $V$ on which $\ell$ is nondegenerate, and there is a basis of $V$ with respect to which the restrictions $\ell|_V$ and $A|_V$ of $\ell$ and $A$ to $V$ are represented by the matrices $\pm H_{\lambda,s_1}$ and $C_{|\lambda|,s_1}$ respectively.
\end{corollary}
\begin{proof}
By Lemma \ref{reducing the general positive eigenvalue case}, there exists an $s_1$-dimensional $A$-invariant space $V$ on which $\ell$ is nondegenerate and there exists a basis of $V$ with respect to which the restrictions $\ell|_V$ and $A|_V$ of $\ell$ and $A$ to $V$ are represented by the matrices $H$ and $J_{\lambda,{s_1}}$, where $H$ is a Hankel matrix satisfying 
\[
H_{i,j}=0\quad\quad\forall\, i+j\leq s_1.
\]
Therefore, by Lemma \ref{normalizing ell for lambda>0}, there is a basis $\{e_1,\ldots,e_{s_1}\}$ of $V$ with respect to which $\ell|_V$ and $A|_V$ are represented by $S_{s_1}$ and $J_{\lambda,s_1}$ respectively. If $\lambda>0$ then this completes the proof because $J_{\lambda,s_1}=C_{|\lambda|,s_1}$. If, on the other hand, $\lambda<0$ then we observe $\ell|_V$ and $A|_V$ are represented by $S_{s_1}$ and $J_{-\lambda,s_1}=C_{|\lambda|,s_1}$ with respect to the basis $\{ie_1,\ldots,ie_{s_1}\}$. So, in either case, we can find a basis with respect to which $\ell|_V$ and $A|_V$ are represented by $H_{\lambda,s_1}=S_{s_1}$ and $C_{|\lambda|,s_1}$.
\end{proof}

For the following proposition, let $r_1,\ldots,r_{n_\lambda}$ and $s_1,\ldots,s_{n_\lambda}$ be the positive integers satisfying $s_i>s_{i+1}$ such that the restriction of $A^2$ to $W_\lambda^{(n)}$ has a Jordan canonical form with $r_i$ Jordan blocks of size $s_i\times s_i$. Note, this definition is consistent with the previous definition of $s_1$, and 
\[
W_\lambda^{(n)}\cong\C^\mu\quad\mbox{ where }\quad{\mu=\sum_{i=1}^{n_\lambda}r_is_i}.
\]
\begin{proposition}\label{real e-val canonical form prop}
There is a basis of $W_\lambda^{(n)}$  with respect to which the restrictions of $\ell$ and $A$ to $W_\lambda^{(n)}$ are represented by the matrices
\[
\bigoplus_{i=1}^{n_\lambda}\left(\bigoplus_{j=1}^{r_i}\epsilon_{i,j}H_{\lambda,s_i}\right)\quad\mbox{ and }\quad \bigoplus_{i=1}^{n_\lambda}\left(\bigoplus_{j=1}^{r_i}C_{|\lambda|,s_i}\right)\quad\mbox{ where }\epsilon_{i,j}=\pm1
\]
respectively.
\end{proposition}
\begin{proof}
By Corollary \ref{C and ell normalized dim 1 eSpace}, there is a space $V\subset W_\lambda^{(n)}$ that is $A$-invariant and  $\ell$-nondegenerate on which $\ell$ and $A$  can be represented by matrices of the desired form. By Lemma \ref{ell-orthogonal complement}, we can normalize $\ell$ and $A$ on $V$ and the $\ell$-orthogonal complement $V^{\perp_\ell}$ of $V$ separately, so we can repeat this process, applying Corollary \ref{C and ell normalized dim 1 eSpace} to $V^{\perp_\ell}$ rather than $W_\lambda^{(n)}$. Iterating the process $\sum_{i=1}^{n_\lambda}r_i$ times completes  the normalization.
\end{proof}

\subsection{Treating Generalized Eigenspaces with Eigenvalue Zero}\label{zero eigenvalue subsection} 
In this subsection we construct a canonical form for the restrictions of $\ell$ and $A$ to the space $W_0^{(n)}$. Our approach is the same as in the proof of Theorem 4.5 in \citep{porter2017absolute}. 
\begin{proposition}\label{zero e-val canonical form prop}
The space $W_0^{(n)}$ can be decomposed into $A$-invariant, pairwise $\ell$-orthogonal subspaces such that there exists a basis with respect to which the restrictions of $\ell$ and $A$ to the decomposition's component subspaces are represented by matrices of the form $C_{0,k}$ and $H_{0,k}$.
\end{proposition}
\begin{proof}
Let 
\[
k=\min \left\{ 
m\in\N\,:\,\left(A|_{W_\lambda^{(n)}}\right)^m\equiv0
\right\}.
\]
Fix a basis, and let $H$ and $C$ be matrices representing $\ell$ and $A$ with respect to this basis.  If $k$ is odd, then $A^{k-1}$ is $\ell$-self-adjoint linear, which implies $HC^{k-1}$ is Hermitian, and hence there is a basis with respect to which the mapping 
\begin{align}\label{HA^k-1 map}
v\mapsto H\left(A^{k-1}v\right)
\end{align}
is represented by a nonzero diagonal matrix. If, on the other hand, $k$ is even, then $A^{k-1}$ is $\ell$-self-adjoint antilinear, which implies $HC^{k-1}$ is symmetric. By Takagi's theorem  in \citep[Theorem 2]{takagi1924algebraic}, for every symmetric matrix $S$, there exists an invertible matrix $U$ such that $US\overline{U}^{-1}$ is diagonal, and, since the map in \eqref{HA^k-1 map} is antilinear whenever $k$ is even, Takagi's theorem implies that there is a basis with respect to which the mapping in \eqref{HA^k-1 map} is represented by a nonzero diagonal matrix.

For either parity of $k$, these observations imply that there exists a vector $a_1\neq0$ such that
\begin{align}\label{top vector choice zero e-space}
H\left(A^{k-1}a_1\right)=\gamma e^{i\theta}a_1\quad\quad\mbox{ for some } \theta\in\R, \gamma>0.
\end{align}
Furthermore, if $k$ is odd then $\theta\equiv0\pmod{\pi}$ because \eqref{HA^k-1 map} is a linear operator represented by a Hermitian matrix. Accordingly, for $z\in\C$,
\[
\ell(A^{k-1}za_1,za_1)=(HA^{k-1}za_1,za_1)=
\begin{cases}
\pm\gamma  |z|^2 \|a_1\|^2&\mbox{ if $k$ is odd}\\
\gamma e^{i\theta}\overline{z}^2\|a_1\|^2&\mbox{ if $k$ is even}.
\end{cases}
\]
Therefore
\[
\ell\left(A^{k-1} \frac{1}{\sqrt{\gamma\|a_1\|^2}  e^{i\theta/2}}a_1,\frac{1}{\sqrt{\gamma\|a_1\|^2}  e^{i\theta/2}}a_1\right)=\pm1.
\]
Define
\[
\tilde e_i=A^{i-1} \frac{1}{\sqrt{\gamma\|a_1\|^2}  e^{i\theta/2}}a_1,
\]
and define
\[
e_1= \tilde e_1+\alpha_2\tilde e_2+\ldots+\alpha_k  \tilde e_k\quad \mbox{ and }\quad e_i=A^{i-1}e_1,
\]
where the coefficients $\alpha_2,\ldots, \alpha_{k}$ are chosen below.
For all $i+j>k+1$ we have 
\[
\ell(e_i,e_j)=\ell(A^{i-1}e_1,A^{j-1}e_1)=\ell(A^{i+j-2}e_1,e_1)=\ell(0,e_1)=0.
\]
Fix the coefficients $\alpha_2,\ldots, \alpha_{k}$ such that for all $j<k$ we have 
\[
\ell(e_1,e_j)=0.
\]
Since $A$ is $\ell$-self-adjoint, we have
\[
\ell(e_{i},e_{i+j})=\ell(e_{i+j},e_{i}),
\]
so our choices of $\alpha_2,\ldots,\alpha_k$ ensure
\[
\ell(e_i,e_j)=0\quad\quad\forall i+j<k+1.
\]
By construction, 
\[
\ell(e_1,e_k)=\ell(e_j,e_{k+j-1})=1,
\]
so the restrictions of $\ell$ and $A$ to the subspace $\text{span}_\C\{e_1,\ldots, e_k\}$ are represented by $S_k$ and $J_{0,k}$ respectively with respect to the basis $\{e_k,\ldots, e_1\}$. 

By Lemma \ref{ell-orthogonal complement}, we can normalize $\ell$ and $A$ on $\text{span}_\C\{e_1,\ldots, e_k\}$ and the orthogonal complement of $\text{span}_\C\{e_1,\ldots, e_k\}$ separately, so this normalization proceedure can be repeated on the orthoganal complement of $\text{span}_\C\{e_1,\ldots, e_k\}$ until $W_{0}^{(n)}$ is exhausted.
\end{proof}

\subsection{Treating Generalized Eigenspaces with Negative Eigenvalues}\label{negative eigenvalue subsection}
Throughout this subsection we assume $\lambda^2<0$ and that the restriction of $A^2$ to $W_\lambda^{(n)}$ has a Jordan canonical form with $2r_i$ Jordan blocks of size $s_i\times s_i$, where $r_1,\ldots,r_{n_\lambda}$ and $s_1,\ldots,s_{n_\lambda}$ are positive integers satisfying $s_i>s_{i+1}$. 

\begin{proposition}\label{negative e-val canonical form prop}
There is a basis of $W_\lambda^{(n)}$  with respect to which the restrictions of $\ell$ and $A$ to $W_\lambda^{(n)}$ are represented by the matrices
\begin{align}\label{negative e-val canonical form}
\bigoplus_{i=1}^{n_\lambda}\left(\bigoplus_{j=1}^{r_i}\epsilon_{i,j}H_{\lambda,s_i}\right)\quad\mbox{ and }\quad \bigoplus_{i=1}^{n_\lambda}\left(\bigoplus_{j=1}^{r_i}C_{\lambda,s_i}\right)\quad\mbox{ where }\epsilon_{i,j}=\pm1
\end{align}
respectively.
\end{proposition}

\begin{proof}

Given the Gohberg--Lancaster--Rodman canonical form for $\ell$ and $A^2$ summarized in Remark \ref{Gohberg--Lancaster--Rodman}, there exists a vector $a_1\in W_\lambda^{(n)}$  such that the restrictions of $\ell$ and $A^2$ to the $s_1$-dimensional vector space $\text{span}_\C\{a_1,(A^2-\lambda^2I)a_1,\ldots, (A^2-\lambda^2I)^{s_1-1}a_1\}$ are represented respectively by $S_{s_1}$ and $J_{\lambda^2,s_1}$ with respect to the basis $\{(A^2-\lambda^2I)^{s_1-1}a_1,(A^2-\lambda^2I)^{s_1-2}a_1,\ldots, a_1\}$.

Defining
\[
a_{k+1}=(A^2-\lambda^2I)a_k\quad\mbox{ and }\quad b_k=Aa_k,
\]
and letting $e_1=\alpha a_1+ \beta b_1$, we have
\begin{align*}
\ell\left(e_1,A(A^2-\lambda^2I)^{s_1-1}e_1\right)&=\alpha^2\ell(a_1,b_{s_1})+\alpha\beta\big(\ell(a_1,Ab_{s_1})+\ell(b_1,b_{s_1})\big)+\beta^2\ell(b_1,Ab_{s_1} \\
&=\alpha^2\ell(a_1,b_{s_1})+\alpha\beta\big(\lambda^2\ell(a_1,a_{s_1})+\lambda^2\ell(a_{s_1},a_{1})\big)+\lambda^2\beta^2\ell(b_1,a_{s_1}) \\
&=\alpha^2\ell(a_1,b_{s_1})\pm2\lambda^2\alpha\beta+\lambda^2\beta^2\ell(b_1,a_{s_1}).
\end{align*}
Clearly, either $\ell(a_1,b_{s_1})=0$ or we can choose $\alpha,\beta\in\C$ such that $(\alpha,\beta)\neq(0,0)$ and 
\[
\ell\left(e_1,A(A^2-\lambda^2I)^{s_1-1}e_1\right)=0.
\]
Accordingly, we can assume, by possibly replacing $a_1$ with $e_1$ as defined above, that 
\[
\ell(a_1,b_{s_1})=0.
\]
With this assumption made, we proceed with $e_1$ defined as above, and will determine the coefficients $\alpha$ and $\beta$ later. Note, this assumption implies also that $\ell(b_1,a_{s_1})=0$ because $(A^2-\lambda^2I)$ is an $\ell$-self-adjoint linear operator, and hence
\[
\ell(e_1,A(A^2-\lambda^2I)^{s_1-1}e_1)=\pm2\lambda^2\alpha\beta.
\]

Define
\[
e_k:=(A^2-\lambda^2I)^{k-1}e_1\quad\mbox{ and }\quad e_{s_1+k}:=Ae_k\quad\quad\forall 1\leq k\leq s_1,
\]
and, on the span of $\{e_i\}$, let $\ell$ and $A$  be represented with respect to the basis  $\{e_{s_1},\ldots,e_1, e_{2s_1},\ldots,e_{s_1+1}\}$ by the matrices
\[
H=
\left(\begin{array}{cc}
H_{1,1} &H_{1,2}\\
H_{2,1}&H_{2,2}
\end{array}
\right)
\quad\mbox{ and }\quad
C=
\left(\begin{array}{cc}
0&J_{\lambda^2,s_1}\\
I_{s_1}&0
\end{array}
\right)
\]
where the each $H_{i,j}$  is an $s_1\times s_1$ matrix. The matrices $H_{i,j}$ are Hankel because $A^{2}-\lambda^2I$ is $H$-self-adjoint. That is, $H_{1,1}$ is Hankel because
\begin{align*}
\ell(e_i,e_{j})=\ell\left((A^2-\lambda^2)^{i-1} e_1,(A^2-\lambda^2)^{j-1} e_1\right)&=\ell\left(e_1,(A^2-\lambda^2)^{i+j-2} e_1\right)\\
&=\ell(e_1,e_{i+j-1})\quad\quad\quad\quad\quad\quad\forall  i+j\leq s_1+1
\end{align*}
and
\begin{align*}
\ell(e_i,e_{j})=\ell\left((A^2-\lambda^2)^{i-1} e_1,(A^2-\lambda^2)^{j-1} e_1\right)&=\ell\left(e_1,(A^2-\lambda^2)^{i+j-2} e_1\right)\\
&=\ell\left(e_1,0\right)\quad\quad\quad\quad\quad\quad\forall  i,j\leq s_1\mbox{ with }i+j> s_1+1.
\end{align*}
Similarly, using the identity $\left((A^2-\lambda^2 I) v,w\right)=\left( v,(A^2-\lambda^2 I)w\right)$, we can show $H_{1,2}$, $H_{2,1}$, and $H_{2,2}$ are Hankel.

Since $A$ is $\ell$-self-adjoint, $HC$ is symmetric, which, as in Lemma \ref{zero diagonals in H lemma}, implies that the $(i,j)$ entry of $H_{2,1}$ is 0 for all $i+j<s_1+1$.  On the other hand, if $s_1+1<i+j$ then still the $(i,j)$ entry of $H_{2,1}$ is 0 because
\[
\ell\left((A^2-\lambda^2)^{i-1}e_1,(A^2-\lambda^2)^{j-1}Ae_1\right)=\ell\left((A^2-\lambda^2)^{i+j-2}e_1,Ae_1\right)=\ell\left(0,Ae_1\right)=0.
\]
Therefore,
\[
H_{1,2}=H_{2,1}=\ell\left(e_1,e_{2s_1}\right)S_{s_1}.
\]
The same analysis shows that the lower left and upper right $s_1\times s_1$ blocks of the matrix representing $\ell$ with respect to the basis $\{a_{s_1},\ldots, a_{1},b_{s_1},\ldots, b_{1}\}$ are also multiples of $S_{s_1}$, that is,
\[
\ell(a_i,b_j)=\ell(a_1,b_{s_1})\delta_{i+j,s_1+1}.
\]
Direct computation also shows that $H_{1,1}$ is a multiple of $S_{s_1}$, that is,
\[
H_{1,1}=\ell(e_1,e_{s_1})S_{s_1},
\]
where 
\begin{align}\label{neg e-val product b}
\ell(e_1,e_{s_1})&=|\alpha|^2\ell(a_1,a_{s_1})+\alpha\overline{\beta}\ell(a_1,b_{s_1})+\beta\overline{\alpha}\ell(b_1,a_{s_1})+|\beta|^2\ell(b_1,b_{s_1}) \nonumber\\
&=\left(|\alpha|^2+\lambda^2|\beta|^2\right)\ell(a_1,a_{s_1})+\alpha\overline{\beta}\ell(a_1,b_{s_1})+\overline{\alpha\overline{\beta}\ell(a_1,b_{s_1})} \nonumber\\
&=\pm \left(|\alpha|^2+\lambda^2|\beta|^2\right)+\alpha\overline{\beta}\ell(a_1,b_{s_1})+\overline{\alpha\overline{\beta}\ell(a_1,b_{s_1})}\nonumber\\
&=\pm \left(|\alpha|^2+\lambda^2|\beta|^2\right).
\end{align}
Since $HC$ is symmetric, it follows that 
\begin{align}
H_{2,2}=\Bigl(H_{1,1}J_{\lambda^2,s_1}\Bigr)^T&=\ell(e_1,e_{s_1})S_{s_1}J_{\lambda^2,s_1}.\nonumber
\end{align}
Lastly, fixing $\alpha=1$ and $\beta=0$, the matrices 
\[
H=
\pm\left(\begin{array}{cc}
S_{s_1} & 0\\
0 &S_{s_1}J_{\lambda^2,s_1}
\end{array}
\right)
\quad\mbox{ and }\quad
C=
\left(\begin{array}{cc}
0&J_{\lambda^2,s_1}\\
I_{s_1}&0
\end{array}
\right)
\]
represent the restrictions of $\ell$ and $A$ on $\text{span}_\C\{e_i\}_{1\leq i\leq 2s_1}$ with respect to a permutation of the basis $\{e_i\}_{1\leq i\leq 2s_1}$.
Since $H$ is nonsingular, By Lemma \ref{ell-orthogonal complement}, we can repeat this construction on the $\ell$-orthogonal complement of $\text{span}_{\C}\{e_1,\ldots, e_{2s_1}\}$, and hence there exists a basis of  $W_\lambda^{(n)}$ with respect to which $\ell$ and $A$ are represented by  the matrices
\begin{align}\label{negative e-val canonical form b}
\bigoplus_{i=1}^{n_\lambda}\left(\bigoplus_{j=1}^{r_i}
 \epsilon_{i,j}
 \left(\begin{array}{cc}
S_{s_i} & 0\\
0 &S_{s_i}J_{\lambda^2,s_i}
\end{array}
\right)
\right)
\quad\mbox{ and }\quad 
\bigoplus_{i=1}^{n_\lambda}\left(\bigoplus_{j=1}^{r_i}
\left(\begin{array}{cc}
0&J_{\lambda^2,s_i}\\
I_{s_i}&0
\end{array}
\right)
\right)
\end{align}
where $\epsilon_{i,j}=\pm1$. In particular, we have shown that if there is a basis with respect to which $\ell$ and  $A$ are represented by 
\begin{align}\label{negative e-val canonical form c}
\bigoplus_{i=1}^{n_\lambda}\left(\bigoplus_{j=1}^{r_i}
 \epsilon_{i,j}
S_{2s_i}
\right)
\quad\mbox{ and }\quad 
\bigoplus_{i=1}^{n_\lambda}\left(\bigoplus_{j=1}^{r_i}
\left(\begin{array}{cc}
0&J_{\lambda^2,s_i}\\
I_{s_i}&0
\end{array}
\right)
\right)
\end{align}
then there is a basis with respect to which $\ell$ and $A$ are represented by the matrices in \eqref{negative e-val canonical form b}, and hence, noting \eqref{transformation}, there exist a matrix $M$ such that 
\begin{align}
(M^{-1})^*\bigoplus_{i=1}^{n_\lambda}\left(\bigoplus_{j=1}^{r_i}
\epsilon_{i,j}
\left(\begin{array}{cc}
S_{s_i} & 0\\
0 &S_{s_i}J_{\lambda^2,s_i}
\end{array}
\right)
\right)M^{-1}=\bigoplus_{i=1}^{n_\lambda}\left(\bigoplus_{j=1}^{r_i}
\epsilon_{i,j}
H_{\lambda,{s_i}}
\right)
\end{align}
and
\begin{align}
M\bigoplus_{i=1}^{n_\lambda}\left(\bigoplus_{j=1}^{r_i}
\left(\begin{array}{cc}
0&J_{\lambda^2,s_i}\\
I_{s_i}&0
\end{array}
\right)
\right)
\overline{M}^{-1}=
\bigoplus_{i=1}^{n_\lambda}\left(\bigoplus_{j=1}^{r_i}
C_{\lambda,s_i}
\right),
\end{align}
which completes the proof.
\end{proof}

\subsection{Treating Generalized Eigenspaces with Nonreal Eigenvalues}\label{noreal eigenvalue subsection}
Throughout this subsection we assume $\lambda^2\not\in\R$ and that the restriction of $A^2$ to $W_\lambda^{(n)}$ has a Jordan canonical form with $2r_i$ Jordan blocks of size $s_i\times s_i$, where $r_1,\ldots,2r_{n_\lambda}$ and $s_1,\ldots,s_{n_\lambda}$ are positive integers satisfying $s_i>s_{i+1}$.

\begin{proposition}\label{nonreal e-val canonical form prop}
There is a basis of $W_\lambda^{(n)}$  with respect to which the restrictions of $\ell$ and $A$ to $W_\lambda^{(n)}$ are represented by the matrices
\begin{align}\label{nonreal e-val canonical form}
\bigoplus_{i=1}^{n_\lambda}\left(\bigoplus_{j=1}^{r_i}\epsilon_{i,j}H_{\lambda,s_i}\right)\quad\mbox{ and }\quad \bigoplus_{i=1}^{n_\lambda}\left(\bigoplus_{j=1}^{r_i}C_{\lambda,s_i}\right)\quad\mbox{ where }\epsilon_{i,j}=\pm1
\end{align}
respectively.
\end{proposition}

\begin{proof}
Given the Gohberg--Lancaster--Rodman canonical form for $\ell$ and $A^2$ summarized in Remark \ref{Gohberg--Lancaster--Rodman}, there exist  vectors $a_1,a^\prime_1\in W_\lambda^{(n)}$  such that the restrictions of $\ell$ and $A^2$ to the $2s_1$-dimensional vector space $\text{span}_\C\{a_1,(A^2-\lambda^2I)a_1,\ldots, (A^2-\lambda^2I)^{s_1-1}a_1,a^\prime_1,\ldots, (A^2-\overline{\lambda}^2I)^{s_1-1}a^\prime_1\}$ are represented respectively by $S_{2s_1}$ and $J_{\lambda^2,s_1}\oplus J_{\overline{\lambda}^2,s_1}$ with respect to the basis $\{ (A^2-\lambda^2I)^{s_1-1}a_1, \ldots,a_1, (A^2-\overline{\lambda}^2I)^{s_1-1}a^\prime_1,\ldots,a^\prime_1\}$. Define
\[
a_{k+1}=(A^2-\lambda^2I)a_k\quad\mbox{ and }\quad a^\prime_{k+1}=\left(A^2-\overline{\lambda}^2I\right)a^\prime_k.
\]
Our goal is to show that there exists a choice of vector $a_1$ such that $\text{span}_\C\{a_1,(A^2-\lambda^2I)a_1,\ldots, (A^2-\lambda^2I)^{s_1-1}a_1,a^\prime_1,\ldots, (A^2-\overline{\lambda}^2I)^{s_1-1}a^\prime_1\}$ is $A$-invariant, so let us proceed assuming otherwise and find a new choice for $a_1$ that satisfies this property.

Define
\[
b_k:=Aa^\prime_k\quad\mbox{ and }\quad b^\prime_k:=Aa_k.
\]
For $1\leq i,j\leq s_1$,
\[
\ell(b_i,b_j)=\ell\left(\overline{\lambda}^2a^\prime_j+a^\prime_{j+1},a^\prime_i\right)=0\quad\mbox{ and }\quad \ell(b^\prime_i,b^\prime_j)=\ell(\lambda^2a_j+a_{j+1},a_i)=0,
\]
and
\[
\ell(b_i,b^\prime_j)=\ell\left(\lambda^2a_j+a_{j+1},a^\prime_i\right) \quad\mbox{ and }\quad  \ell(b^\prime_i,b_j)=\ell\left(\overline{\lambda}^2a^\prime_j+a^\prime_{j+1},a_i\right).
\]
Therefore, the restrictions of $\ell$ and $A^2$ to the $2s_1$-dimensional vector space $\text{span}_\C\{b_1,\ldots, b_{s_1},b^\prime_1,\ldots, b^\prime_{s_1}\}$ are represented respectively by
\[
\left(
\begin{array}{cc}
0 & J_{\lambda^2,s_1}S_{s_1}\\
J_{\overline{\lambda}^2,s_1}S_{s_1} & 0
\end{array}
\right)
\quad\mbox{ and }\quad J_{\lambda^2,s_1}\oplus J_{\overline{\lambda}^2,s_1}
\]
with respect to the basis $\text{span}_\C\{b_1,\ldots, b_{s_1},b^\prime_1,\ldots, b^\prime_{s_1}\}$.

Letting $e_1=\alpha a_1+ \beta b_1$, we have
\begin{align}\label{neg e-val product a}
\ell\left(e_1,A(A^2-\lambda^2I)^{s_1-1}e_1\right)&=\alpha^2\ell(a_1,b^\prime_{s_1})+\alpha\beta\big(\ell(a_1,A^2a^\prime_{s_1})+\ell(b_1,b^\prime_{s_1})\big)+\beta^2\ell(b_1,A^2a^\prime_{s_1})\nonumber\\
&=\alpha^2\ell(a_1,b^\prime_{s_1})+\alpha\beta\big(\lambda^2\ell(a_1,a^\prime_{s_1})+ \overline{\lambda}^2\ell(a^\prime_{s_1},a_1)\big)+ {\lambda}^2\beta^2\ell(b_1,a^\prime_{s_1})\nonumber\\
&=\alpha^2\ell(a_1,b^\prime_{s_1})+\alpha\beta \left(\lambda^2 + \overline{\lambda}^2 \right)+ {\lambda}^2\beta^2\ell(b_1,a^\prime_{s_1}).
\end{align}
Define
\[
e_k=(A^2-\lambda^2)^{k-1}e_1\quad\mbox{ and }\quad e_{s_1+k}=Ae_k\quad\quad\forall 1\leq k\leq s_1
\]
and, on the span of $\{e_i\}$, let $\ell$ and $A$  be represented with respect to the basis  $\{e_i\}$ by the matrix
\[
H=
\left(\begin{array}{cc}
H_{1,1} &H_{1,2}\\
H_{2,1}&H_{2,2}
\end{array}
\right)
\quad\mbox{ and }\quad
C=
\left(\begin{array}{cc}
0&J_{\lambda^2,s_1}\\
I_{s_1}&0
\end{array}
\right)
\]
where the matrices $H_{i,j}$ are each $s_1\times s_1$. Direct computation yields $H_{1,1}=H_{2,2}=0$. Furthermore, $H_{1,2}$ is symmetric because $HC$ is symmetric, and hence $H_{2,1}$  is symmetric as well. If $s_1+1<i+j$ then the $(i,j)$th entry of $H_{1,2}$ is zero because
\[
\ell\left((A^2-\lambda^2)^{i-1}e_1,(A^2-\overline{\lambda}^2)^{j-1}Ae_1\right)=\ell\left((A^2-\lambda^2)^{i+j-2}e_1,Ae_1\right)=\ell\left(0,Ae_1\right)=0.
\]
Accordingly,
\[
\det(H)=\pm\big(|\ell(e_1,e_{2s_1})|\big)^{2s_1},
\]
which, by \eqref{neg e-val product a}, can be made nonzero for an adequate choice of $\alpha$ and $\beta$. Since $HC$ is symmetric, evaluating the lower right $s_1\times s_1$ block of $HC$ yields, as in Lemma \ref{zero diagonals in H lemma}, that 
\[
H_{2,1}=\overline{H_{1,2}}=\ell(e_1,e_{2s_1})S_{s_1}.
\]
By replacing $e_1$ with $\sqrt{\frac{1}{\ell(e_1,e_{2s_1})}}e_1$, we can assume $\ell(e_1,e_{2s_1})=\pm1$, so the restrictions of $\ell$ and $A$ to $\text{span}_\C\{e_1,\ldots,e_{2s_1}\}$ are represented by the matrices 
\[
\pm S_{2s_1}
\quad\mbox{ and }\quad 
\left(\begin{array}{cc}
0&J_{\lambda^2,s_1}\\
I_{s_1}&0
\end{array}
\right)
\]
with respect to the basis $\{e_1,\ldots,e_{2s_1}\}$. By Lemma \ref{ell-orthogonal complement}, we can repeat this normalization proceedure on the $\ell$-orthogonal complement of $\text{span}_\C\{e_1,\ldots,e_{2s_1}\}$, and hence there is a basis of $W_{\lambda}^{(n)}$ with respect to which $\ell$ and $A$ are represented by the matrices 
\begin{align}\label{nonreal e-val canonical form b}
\bigoplus_{i=1}^{n_\lambda}\left(\bigoplus_{j=1}^{r_i}
\epsilon_{i,j}S_{2s_i}
\right)
\quad\mbox{ and }\quad 
\bigoplus_{i=1}^{n_\lambda}\left(\bigoplus_{j=1}^{r_i}
\left(\begin{array}{cc}
0&J_{\lambda^2,s_i}\\
I_{s_i}&0
\end{array}
\right)
\right)
\quad\quad\mbox{where $\epsilon_{i,j}=\pm1$.}
\end{align}
\end{proof}

\subsection{A Canonical Form for Antilinear Operators}\label{A canonical form for antilinear operators} It is worth noting that methods applied above can be used to obtain the canonical form for antilinear operators (without considering Hermitian forms) given in \cite[Theorem 3.1]{hong1988canonical}, referred to in Remark \ref{Hong--Horn}, so here we briefly outline how this is done.

On a generalized eigenspace $W_{\lambda}^{(n)}$ for which $\lambda\not\in\R$, in subsections \ref{negative eigenvalue subsection} and \ref{noreal eigenvalue subsection} we normalize the restriction $A|_V$ of $A$ to a subspace $V$, where $V$ is defined to be the space spanned by some Jordan chain of $A^2$ and the image of $A$ applied to this Jordan chain, and achieve the normalization by first choosing a basis with respect to which $A^2$ has the Jordan normal form and then transforming this basis to a new one with respect to which $A$ has the form in Theorem \ref{simultaneous canonical form theorem}, all the while taking care to simultaneously normalize $\ell$. The very same procedure can be applied to normalize $A|_V$ without the additional steps needed to normalize $\ell$, that is, one can normalize $A|_V$ by reading through the proofs of propositions \eqref{negative e-val canonical form prop} and \eqref{nonreal e-val canonical form prop} while disregarding all mention of $\ell$ (e.g., using the Jordan normal form rather than the Gohberg--Lancaster--Rodman form). Next, letting $U$ denote the $A$-invariant space on which we have already normalized $A$, we repeat this normalization on any $A$-invariant subspace of $W_{\lambda}^{(n)}\setminus U$ containing a maximal Jordan chain of $A^2$ rather than applying Lemma \ref{ell-orthogonal complement} to choose a specific $A$-invariant complement of $U$. To find such a subspace, we choose any maximal length Jordan chain of $A^2$ in $W_{\lambda}^{(n)}\setminus U$ and consider the subspace spanned by this chain and the image of $A$ applied to this chain.

On the generalized eigenspace $W_{0}^{(n)}$, we may normalize the restriction $A|_V$ of $A$ to a subspace $V$, where $V$ is a maximal subspace of $W_{0}^{(n)}$ that has a basis obtained by applying powers of $A$ to a single vector, by using the procedure in the proof of Proposition \ref{zero e-val canonical form prop}, again disregarding all mention of $\ell$, that is, rather than choosing $a_1\in\{v\,|\, A^{k-1}v\neq0\}\cap W_{0}^{(n)}$ such that \eqref{top vector choice zero e-space} holds we simply choose $a_1$ to be an arbitrary vector in $\{v\,|\, A^{k-1}v\neq0\}\cap W_{0}^{(n)}$. We repeat this normalization on any maximal $A$-invariant subspace of $W_{0}^{(n)}\setminus U$ (where $U$ denotes the $A$-invariant  space on which we have already normalized $A$) that has a basis obtained by applying powers of $A$ to a single vector. To find such a subspace, we choose any vector $v\in W_{0}^{(n)}\setminus U$ for which the subspace spanned by $\{v,Av,\ldots,A^nv\}$ has maximal dimension.

Lastly, on a generalized eigenspace $W_{\lambda}^{(n)}$ for which $\lambda^2>0$, we apply Lemma \ref{reducing the general positive eigenvalue case} to normalize the restriction $A|_V$ of $A$ to a subspace $V$, where $V$ is the span of a Jordan chain of $A^2$ given by Lemma \ref{reducing the general positive eigenvalue case}. Note, the proof of Lemma \ref{reducing the general positive eigenvalue case} does not use the assumption that $A$ is $\ell$-self-adjoint for some Hermitian form $\ell$. And as in the previous two cases, we repeat this normalization on any $A$-invariant  subspace of $W_{\lambda}^{(n)}\setminus  U$ (where, again, $U$ denotes the $A$-invariant space on which we have already normalized $A$) containing a maximal length Jordan chain of $A^2$.

Given that every antilinear operator can be represented by a matrix representing the antilinear operator of a pair in the canonical form of Theorem \ref{simultaneous canonical form theorem}, we have the following lemma.
\begin{lemma}\label{every operator is self-adjoint}
Every antilinear operator on $\C^n$ is $\ell$-self-adjoint with respect to some nondegenerate Hermitian form $\ell$.
\end{lemma}

\section{Alternative Canonical Forms}\label{alter_section}

We conclude this text with a few remarks regarding an alternative approach to deriving a canonical form for the pair $(\ell, A)$, and we record an alternative canonical form, Theorem \ref{alt simultaneous canonical form theorem}, that naturally arises from this approach. The form in Theorem \ref{simultaneous canonical form theorem} has some advantages. Its matrices have a minimal number of nonzero entries, for example. The form in Theorem \ref{alt simultaneous canonical form theorem} is, however, better suited for certain applications. Namely, analysis involving antilinear operators often includes consideration of the operators' squares, making use of the squares' linearity and  well developed theory for linear operators. The alternative canonical forms of Theorems \ref{alt simultaneous canonical form theorem} and \ref{canonical form theorem} below are ideal for studying $A$ and $A^2$ simultaneously because $A^2$ is represented by a Jordan matrix whenever $A$ is represented by the canonical form of Theorem \ref{canonical form theorem}.

When searching for a canonical form for $(\ell, A)$, after noticing that a linear operator $A^2$ is $\ell$-self-adjoint whenever the antilinear operator $A$ is $\ell$-self-adjoint, it becomes natural to apply the Gohberg--Lancaster--Rodman form to the pair $(\ell, A^2)$. Specifically, one may try to normalize $(\ell, A)$ by bringing $(\ell, A^2)$ to the Gohberg--Lancaster--Rodman form and then changing the basis to normalize $A$ while tracking the changes induced in the matrix representing $\ell$ (ideally, one would like to achieve this without changing the matrix representing $\ell$ at all). Indeed, we use this approach in subsections  \ref{negative eigenvalue subsection} and \ref{noreal eigenvalue subsection}, and, from this perspective, noting Lemma \ref{GLR invariant lemma}, one must wonder why we do not use this approach in section \ref{positive eigenvalue subsection} as well. It turns out to be absolutely viable for the normalization carried out in section \ref{positive eigenvalue subsection}, but the method presented in section \ref{positive eigenvalue subsection} is simply more efficient. Applying this alternative approach to carry out the normalization has its own merit, however, because it naturally leads one to discover the canonical form given in Theorem \ref{alt simultaneous canonical form theorem} below. 

To explore this further, let us consider the special case wherein $A^2:\C^n\to\C^n$ has a single eigenvalue $\lambda^2$, its only eigenspace is 1-dimensional, and $\lambda^2>0$  (note, applying Lemmas \ref{ell-orthogonal complement} and \ref{reducing the general positive eigenvalue case}, one can always reduce to this special case for the normalization carried out in section \ref{positive eigenvalue subsection}). Applying the Gohberg--Lancaster--Rodman form to the pair $(\ell, A^2)$, we can choose a basis of $\C^n$ with respect to which $\ell$ and $A$ are represented by matrices $S_n$ and $C$ respectively such that
\begin{align}\label{form of CBarC}
C\overline{C}=J_{\lambda^2,n}.
\end{align}
We attempt to normalize $A$ by changing the basis with transformations that preserve the matrix representations of $\ell$ and $A^2$. Hence we consider the transformations represented by matrices in the group
\begin{align}\label{symmetry group preserving GLR}
G:=\left\{M\in M_{n\times n}(\C)\,|\, M^*S_{n}M=S_{n}\mbox{ and }MC\overline{C}=C\overline{C}M\right\}
\end{align}
acting on the subspace
\[
\mathcal{C}:=\{C\in M_{n\times n}(\C)\,|\,C\overline{C}=J_{\lambda^2,n}\}
\]
of $GL_{n}(\C)$, via the action $(M,C)\mapsto M C \overline{M}^{-1}$. It turns out that we can solve \eqref{form of CBarC}, that is, we can completely describe the general form of a matrix $C$ satisfying \eqref{form of CBarC}, and $G$ acts transitively on $\mathcal{C}$.\footnote{The space $\mathcal{C}$ turns out to be homeomorphic to the Cartesian product $S^1\times \R^{n-1}$ of a circle and Euclidean space with the product topology.} Matrices in $\eqref{form of CBarC}$ turn out to be upper-triangular and Toeplitz, and, for a matrix $C\in \mathcal{C}$, one can explicitly construct a matrix $M\in G$ such that $M {C} \overline{M}^{-1}\in GL_n(\R)$ and the eigenvalue of $M {C} \overline{M}^{-1}$ is $|\lambda|$.  Choosing $M$ to satisfy these conditions, it turns out that $M {C} \overline{M}^{-1}$ equals the matrix $M_{|\lambda|,n}$ defined below, which confirms that $G$ acts transitively on $\mathcal{C}$. Of course, we have omitted details of the calculations summarized here, but the summary provides an outline of how one can apply the aforementioned alternative approach to the normalization carried out in section \ref{positive eigenvalue subsection}. Furthermore, this summary illustrates how, from one perspective, the alternative canonical form given in Theorem \ref{alt simultaneous canonical form theorem} below arises naturally.

This alternative form features the sequence 
\begin{align}\label{catalan-like sequence}
c_0(\lambda):=\lambda,\quad c_1(\lambda):=\frac{1}{2\lambda},\quad\mbox{ and }\quad c_i(\lambda):=\frac{-1}{2\lambda}\sum_{j=1}^{i-1}c_j(\lambda)c_{i-j}(\lambda),
\end{align}
which arises if we try to solve the matrix equation
\begin{align}\label{square root of Jordan matrix a}
C^2=J_{\lambda^2,k}\quad\quad \lambda\neq 0 
\end{align}
by supposing $C$ has the form 
\begin{align}\label{square root of Jordan matrix b}
C= \sum_{i=1}^{k}c_{i-1}(\lambda)T_k^{i-1}
\end{align}
and comparing coefficients, interpreting each side of the equation as a degree $k-1$ polynomial in $T_k$.\footnote{If $C$ satisfies \eqref{square root of Jordan matrix a} then $C$ satisfies \eqref{square root of Jordan matrix b} for some choice of coefficients $c_i(\lambda)$. Nevertheless, proving this fact is not necessary for understanding the provenance of \eqref{catalan-like sequence}, so we introduce \eqref{square root of Jordan matrix b} as though it is not a consequence of \eqref{square root of Jordan matrix a}.}
An interesting observation is that  the sequence
\[
|c_1(1/2)|=1,|c_2(1/2)|=1,|c_3(1/2)|=2,\ldots
\]
is known as the Catalan numbers, $|c_i(1/2)|= \frac{1}{i+1}\genfrac{(}{)}{0pt}{1}{2i}{i}$, which play an important role in combinatorics. The identity
\[
c_{i}(\lambda)=(-1)^i(2\lambda)^{1-2i}|c_i(1/2)|=\frac{(-1)^{i+1}(2\lambda)^{1-2i}}{i+1}\genfrac{(}{)}{0pt}{0}{2i}{i},
\]
valid for all positive integers $i$, further illuminates the relationship between $\{c_{i}(\lambda)\}_{i=1}^\infty$ and the Catalan numbers.

For $\lambda\in\C$, we define the $k\times k$ or $2k\times 2k$  matrix $M_{\lambda,k}$ by
\[
M_{\lambda,k}:= 
\begin{cases}
\sum_{i=1}^{k}c_{i-1}(\lambda)T_k^{i-1}&\mbox{ if } \lambda\in\R\setminus\{0\}\\
\frac{1}{2}\left(
\begin{array}{cc}
J_{1,\frac{k}{2}}& -J_{-1,\frac{k}{2}}\\
J_{-1,\frac{k}{2}}& -J_{1,\frac{k}{2}}
\end{array}
\right)
&\mbox{ if }\lambda=0\mbox{ and $k$ is even}\\
\left(
\begin{array}{c|c|c}
0&0&  I_{ \frac{k-1}{2}}\\\hline
 I_{ \frac{k-1}{2}}&0&0\\\hline
 0&0 &0
\end{array}
\right)
&\mbox{ if }\lambda=0\mbox{ and $k$ is odd}\\
\left(
\begin{array}{cc}
0&  \sum_{i=1}^{k}c_{i-1}(\lambda)T_k^{i-1}\\
\sum_{i=1}^{k}c_{i-1}(\overline{\lambda})T_k^{i-1}& 0
\end{array}
\right)
&\mbox{ otherwise},
\end{cases}
\]
where  $0$ denotes a matrix of appropriate size with zero in all entries and, for odd $k$, $M_{0,k}$ is a $k\times k$ matrix. 
We define corresponding matrices $N_{\lambda, k}$ by
\[
N_{\lambda,k}:= 
\begin{cases}
S_{k}&\mbox{ if } \lambda\in\R\setminus\{0\}\\
S_{ \frac{k}{2}}  \oplus \left(-S_{ \frac{k}{2}}\right) &\mbox{ if } \lambda= 0 \mbox{ and $k$ is even}\\
S_{\left\lfloor \frac{k}{2}\right\rfloor}  \oplus S_{\left\lceil \frac{k}{2}\right \rceil} &\mbox{ if } \lambda= 0 \mbox{ and $k$ is odd}\\
S_k\oplus (-S_k)&\mbox{ if } \lambda^2 <0\\
S_{2k}&\mbox{ otherwise},
\end{cases}
\]
where $ \lceil a\rceil$ denotes the smallest integer not less than $a$ and $\lfloor a\rfloor$ denotes the largest integer not larger than $a$. For the following theorem, we let $\{\lambda_1,\lambda_2,\ldots,\lambda_\gamma\}$ denote the subset of principle square roots of eigenvalues of $A^2$ enumerated in section \ref{The canonical forms section}. 


\begin{theorem}\label{alt simultaneous canonical form theorem}
The domain of an $\ell$-self-adjoint antilinear operator $A$ can be decomposed into $A$-invariant, pairwise $\ell$-orthogonal subspaces such that there exists a basis with respect to which the restrictions of $\ell$ and $A$ to the decomposition's component subspaces are represented by matrices of the form $\pm N_{\lambda,k}$ and $M_{\lambda,k}$ where $\lambda\in\{\lambda_1,\lambda_2,\ldots,\lambda_\gamma\}$ and $k\in\N$. The corresponding block  diagonal matrices representing $\ell$ and $A$ are unique up to a permutation of the blocks on the diagonal.
\end{theorem}

A canonical form for antilinear operators, described in Remark \ref{Hong--Horn} and section \ref{A canonical form for antilinear operators}, is given by Hong and Horn in \cite[Theorem 3.1]{hong1988canonical}. Since, as is noted in Lemma \ref{every operator is self-adjoint}, every antilinear operator is $\ell$-self-adjoint with respect to some nondegenerate Hermitian form $\ell$, by applying Theorem \ref{alt simultaneous canonical form theorem} to the pair $\ell$ and $A$ to get another matrix representation for $A$, we obtain the following alternative canonical form for antilinear operators.

\begin{theorem}\label{canonical form theorem}
The domain of an antilinear operator $A$ can be decomposed into $A$-invariant subspaces such that there exists a basis with respect to which the restriction of $A$ to the decomposition's component subspaces are represented by matrices of the form $M_{\lambda,k}$ where $\lambda\in\{\lambda_1,\lambda_2,\ldots,\lambda_\gamma\}$ and $k\in\N$. The corresponding block diagonal  matrix representing $A$ is unique up to a permutation of the blocks on the diagonal.
\end{theorem}
\begin{remark}
In a basis with respect to which $A$ is represented by a matrix with the above canonical form, $A^2$ is represented by a Jordan matrix. Similarly, if $\ell$ and $A$ are represented by matrices in the canonical form of Theorem \ref{alt simultaneous canonical form theorem} then the pair $(\ell,A^2)$ is represented by matrices in the Gohberg--Lancaster--Rodman form. Noting this connection together with Lemma \ref{GLR invariant lemma}, one can readily show that if $A$ is nonsingular then Theorems \ref{simultaneous canonical form theorem} and \ref{alt simultaneous canonical form theorem} are indeed equivalent. To show that each of these theorems is a consequence of the other in the more general case where $A$ is singular, it is not too difficult to explicitly construct a basis change of the maximal subspace on which $A$ is nilpotent transforming the canonical form in Theorem \ref{simultaneous canonical form theorem} to the form in Theorem \ref{alt simultaneous canonical form theorem} (and vice versa);
 for example, considering a change of basis transformation represented by 
\[
T:=
\frac{1}{\sqrt{2}}\left(
\begin{array}{cccccc}
    1 & 1 & 0 & 0 & 0 & 0  \\
    0 & 0 & 1 & 1 & 0 & 0  \\
    0 & 0 & 0 & 0 & 1 & 1  \\
    -1 & 1 & 0 & 0 & 0 & 0  \\
    0 & 0 & -1 & 1 & 0 & 0  \\
    0 & 0 & 0 & 0 & -1 & 1
\end{array}
\right),
\]
we have $(T^{-1})^*H_{0,6} T^{-1}=N_{0,6}$ and $TC_{0,6}\overline{T}^{-1}=M_{0,6}$, that is, this change of basis transforms a certain matrix representation given by Theorem \ref{simultaneous canonical form theorem} to a matrix representation given by Theorem \ref{alt simultaneous canonical form theorem}.
\end{remark}

\bibliographystyle{cell}
\bibliography{biblio}

\begin{thebibliography}{}

\bibitem[Asano and Nakayama, 1938]{asano}
Asano, K. and Nakayama, T. (1938{\rm{}}).
\newblock \"{U}ber halblineare Transformationen.
\newblock {\rm Math. Ann.} \emph{115}, 87--114.

\bibitem[Benedetti and Cragnolini, 1984]{benedetti1984simultaneous}
Benedetti, R. and Cragnolini, P. (1984{\rm{}}).
\newblock On simultaneous diagonalization of one Hermitian and one symmetric
  form.
\newblock {\rm Linear Algebra and its Applications } \emph{57}, 215--226.

\bibitem[Cartan, 1922]{cartan}
Cartan, E. (1922{\rm{}}).
\newblock Sur la g\'eom\'etrie pseudo-conforme des hypersurfaces de l'espace de
  deux variables complexes.
\newblock {\rm Ann. Mat. Pura Appl.} \emph{11}, 17--90.

\bibitem[Chern and Moser, 1974]{chern}
Chern, S. and Moser, J.~K. (1974{\rm{}}).
\newblock Real hypersurfaces in complex manifolds.
\newblock {\rm Acta Math.} \emph{133}, 219--271.

\bibitem[Gantmakher, 1953]{gantmakher1953theory}
Gantmakher, F.~R. (1953{\rm{}}).
\newblock {\rm The theory of matrices}, vol. 131,.
\newblock American Mathematical Soc.

\bibitem[Gohberg {\rm et~al.}, 2006]{gohberg2006indefinite}
Gohberg, I., Lancaster, P.  and Rodman, L. (2006{\rm{}}).
\newblock {\rm Indefinite linear algebra and applications}.
\newblock Springer Science \& Business Media.

\bibitem[Haantjes, 1935]{haant}
Haantjes, J. (1935{\rm{}}).
\newblock Klassification der antilinearen Transformationen.
\newblock {\rm Math. Ann.} \emph{112}, 98--106.

\bibitem[Hong, 1985]{hongthesis}
Hong, Y. (1985{\rm{}}).
\newblock {\rm Consimilarity: Theory and Applications}.
\newblock Johns Hopkins Univ.

\bibitem[Hong, 1991]{hong1991}
Hong, Y. (1991{\rm{}}).
\newblock A Canonical Form Under $\varphi$-Equivalence.
\newblock {\rm Linear Algebra and its Applications } \emph{147}, 501--549.

\bibitem[Hong and Horn, 1988]{hong1988canonical}
Hong, Y. and Horn, R.~A. (1988{\rm{}}).
\newblock A canonical form for matrices under consimilarity.
\newblock {\rm Linear Algebra and its Applications } \emph{102}, 143--168.

\bibitem[Hong {\rm et~al.}, 1986]{hong1986reduction}
Hong, Y.~P., Horn, R.~A.  and Johnson, C.~R. (1986{\rm{}}).
\newblock On the reduction of pairs of Hermitian or symmetric matrices to
  diagonal form by congruence.
\newblock {\rm Linear Algebra and its Applications } \emph{73}, 213--226.

\bibitem[Jacobowitz, 1990]{jacobowitz}
Jacobowitz, H. (1990{\rm{}}).
\newblock {\rm An introduction to CR structures}, vol. 38, of {\rm Mathematical
  Surveys and Monographs}.
\newblock American Mathematical Soc.

\bibitem[Jacobson, 1943]{jacobson}
Jacobson, N. (1943{\rm{}}).
\newblock {\rm The Theory of Rings}, vol. II, of {\rm Mathematical Surveys and
  Monographs}.
\newblock American Mathematical Soc.

\bibitem[Porter and Zelenko, 2017]{porter2017absolute}
Porter, C. and Zelenko, I. (2017{\rm{}}).
\newblock Absolute parallelism for 2-nondegenerate CR structures via bigraded
  Tanaka prolongation.
\newblock {\rm arXiv preprint arXiv:1704.03999 } \emph{}.

\bibitem[Takagi, 1924]{takagi1924algebraic}
Takagi, T. (1924{\rm{}}).
\newblock On an Algebraic Problem Related to an Analytic Theorem of
  Carath{\'e}odory and Fdj{\'e}r and on an Allied Theorem of Landau.
\newblock {\rm Proceedings of the Physico-Mathematical Society of Japan. 3rd
  Series } \emph{6}, 130--140.

\bibitem[Tanaka, 1962]{tanakaCR}
Tanaka, N. (1962{\rm{}}).
\newblock On the pseudo-conformal geometry of hypersurfaces of the space of $n$
  complex variable.
\newblock {\rm J. Math. Soc. Japan } \emph{14}, 397--429.

\bibitem[Thompson, 1991]{thompson}
Thompson, R.~C. (1991{\rm{}}).
\newblock Pencils of Complex and Real Symmetric and Skew Matrices.
\newblock {\rm Linear Algebra and its Applications } \emph{147}, 323--371.

\end{thebibliography}
\end{document}